\documentclass [12pt,a4paper,reqno]{amsart}

\usepackage{amsmath,amsfonts,amssymb,amscd,verbatim,delarray,verbatim}

\theoremstyle{plain}
\newtheorem{theorem}{Theorem}[section]\newtheorem{Theorem}{Theorem}[section]

\newtheorem{corollary}[theorem]{Corollary}\newtheorem{Corollary}[theorem]{Corollary}
\newtheorem{prop}[theorem]{Proposition}\newtheorem{Proposition}[theorem]{Proposition}
\newtheorem{Lemma}[theorem]{Lemma}

\theoremstyle{definition}
\newtheorem{defn}[theorem]{Definition}\newtheorem{Definition}[theorem]{Definition}

\newtheorem{remark}[theorem]{Remark}
\newtheorem{Example}[theorem]{Example}

   \newcommand{\ov}{\overline}
\newcommand{\thmref}[1]{Theorem~\ref{#1}}
\newcommand{\secref}[1]{Section~\ref{#1}}

\newcommand{\lemref}[1]{Lemma~\ref{#1}}
\newcommand{\corref}[1]{Corollary~\ref{#1}}

\newcommand{\remarkref}[1]{Remark~\ref{#1}}

\newcommand{\propref}[1]{Proposition~\ref{#1}}

\DeclareMathOperator{\PSL}{PSL} 
\DeclareMathOperator{\PGL}{PGL} \DeclareMathOperator{\GL}{GL}
\DeclareMathOperator{\tr}{tr} 
\DeclareMathOperator{\Tor}{Tor}
\DeclareMathOperator{\TD}{TD}
\DeclareMathOperator{\mult}{mult}\DeclareMathOperator{\rk}{rk}
\DeclareMathOperator{\invl}{inv}
\DeclareMathOperator{\codim}{codim} 
\DeclareMathOperator{\Aut}{Aut} 
\DeclareMathOperator{\Bim}{Bim} 
\DeclareMathOperator{\NS}{NS} 
\DeclareMathOperator{\Bir}{Bir} \DeclareMathOperator{\Pic}{Pic} 

\newcommand{\BC}{\mathbb{C}}

\newcommand{\BP}{\mathbb{P}}
\newcommand{\BQ}{\mathbb{Q}}
\newcommand{\BA}{\mathbb{A}}\newcommand{\BR}{\mathbb{R}}

\newcommand{\al}{\alpha}
\newcommand{\be}{\beta}

\DeclareMathSymbol{\twoheadrightarrow}  {\mathrel}{AMSa}{"10}
               \def\LL{{\mathcal L}} 
     \def\AA{{\mathcal A}}

\def\TT{{\mathbf T}}

\begin{document}

\title[Bimeromorphic automorphisms  of  $\BP^1-$bundles] {Bimeromorphic automorphism groups  of certain  $\BP^1-$bundles}
% over manifolds ofr  Abelian subgroups of the group $\Bim(X)$ where $X$ is a conic bundle over  torus of algebraic dimension 0.}
\author{Tatiana Bandman }\address{Department of Mathematics,
Bar-Ilan University,
Ramat Gan, 5290002, Israel}
\email{bandman@macs.biu.ac.il}
\author{Yuri  G. Zarhin}
\address{Pennsylvania State University, Department of Mathematics, University Park, PA 16802, USA}
\email{zarhin@math.psu.edu}
\thanks{The second named author (Y.Z.) was partially supported by Simons Foundation Collaboration grant   \# 585711.}

\begin{abstract}

We call a group   $G$  very Jordan if it contains a normal abelian subgroup  $G_0$  such that the orders of  finite subgroups of the quotient  $G/G_0$ are bounded by a constant depending on $G$ only.  Let $Y$ be a complex torus   of algebraic dimension 0.
We prove that  
if $X$ is a non-trivial holomorphic $\BP^1-$bundle  over $Y$  then the group $\Bim(X)$  of its bimeromorphic   automorphisms is very Jordan
(contrary  to the case when $Y$  has positive algebraic dimension).  This assertion remains true if $Y$ is any connected compact complex  K\"{a}hler    manifold of algebraic dimension 0  without rational curves or analytic subsets of codimension 1.

\end{abstract}

\subjclass[2010]{14E05, 14E07,14J50, 32L05, 32M05, 32J27, 32Q15.}
\keywords{Automorphism groups of compact complex manifolds, algebraic dimension 0, complex tori, conic bundles, Jordan properties of groups}
\maketitle

\section {Introduction}\label {intro}

  Let $X$ be a compact  complex connected      manifold. We denote by $\Aut(X)$ and $\Bim(X) $   the groups of automorphisms and  bimeromorphic  selfmaps of $X, $ respectively.
 If $X$ is projective,  $\Bir(X) $ denotes the group of birational automorphisms of $X.$ As usual,  $\BP^n$ stands for the $n$-dimensional complex projective space;  $a(X)$ stands for the algebraic dimension of $X.$  All manifolds in this paper are assumed to be {\bf complex  compact and connected} unless otherwise stated. 
 
V.L. Popov  in  \cite{Pop} defined the Jordan property of a group and raised 
 the following question:
{\it   when the groups $\Aut(X)$ and $ \Bir(X)$  are Jordan?}

 \begin{defn}\label{groups}  {}   \    \begin{itemize}\item     A group $G$ is called {\sl bounded} if the orders of its finite subgroups are bounded by an universal constant that depends only on $G$  (\cite[Definition 2.9]{Pop}).
\item A group $G$ is called {\sl Jordan} if there is a positive integer $J$ such that
every finite subgroup $B$ of $G$ contains an abelian subgroup $A$  that is normal in $B$ and such that the index $[B:A]\le J$ \cite[Definition 2.1]{Pop}.\end{itemize}\end{defn}
% A Jordan group $G$ is called {\sl strongly Jordan} \cite{PS14,BZ17} if there is a positive integer $m$  such that every finite subgroup %of $G$ is generated by at most $m$ elements.
 In this paper we are interested in the following  property of groups.
  \begin{defn}\label{groups1}  We call a group  $G$  {\sl very Jordan} if there exist  a  commutative normal subgroup  $G_0$ of $G$ and a bounded group $F$ that sit in  a short exact sequence\begin{equation}\label{veryjordan}
1\to G_0\to G\to F\to 1.\end{equation}
\end{defn}
\begin{remark}\label{vstavka1} 1) Every finite group is bounded, Jordan, and very Jordan. 

2) Every commutative group is Jordan and very Jordan. 

3) Every finitely generated commutative group is bounded.

4) A subgroup of a very Jordan group is very Jordan.\end{remark}

5) ``Bounded" implies ``very Jordan", ``very Jordan" implies ``Jordan".

The first goal  of the paper is   to find complex manifolds with very Jordan group $\Aut(X)$ or $\Bim(X).$  To this end we prove the following 
generalization of \cite[Lemma  2.5]  {MengZhang},
and  of \cite[ Lemma 3.1]{Kim}.
\begin{Proposition}\label{abounded1}(see \propref{abounded})
Let $X$ be a connected compact complex   K\"{a}hler   manifold  and   $F=\Aut(X)/\Aut_0(X),$ where $\Aut_0(X) $ is  the connected identity component of  $\Aut(X).$   Then  
$F$ is bounded.\end{Proposition}

It follows that  if group $\Aut_0(X) $ is commutative, then $\Aut(X) $  is very Jordan.

  \begin{Example}\label{vvveryjordan}    
   If $X$ is a compact complex   K\"{a}hler   manifold of non-negative Kodaira dimension, then $\Aut(X)$ is very Jordan (\cite[Proposition 5.11]{Fu78} and \corref{vj}   below).\end{Example}

In this paper we study another wide and interesting class of complex manifolds with very Jordan group of automorphisms, namely, compact uniruled manifolds that are equidimensional rational  fibrations  (i.e. all components of all the fibers are one-dimensional and the general fiber is $\BP^1$)
over complex tori of algebraic dimension zero. 

In order to demonstrate the role of such manifolds,   we want to survey Jordan  properties of  $\Aut(X)$ and $\Bim(X)$ for various types of  compact  complex manifolds   $X.$

 The group $\Aut(X)$ of any connected complex compact manifold $X$ carries a natural structure of a complex (not necessarily connected) Lie group such that the action map $\Aut(X)\times X \to X$ is holomorphic (Theorem of  Bochner-Montgomery,  \cite{BM}).  It is known, for example, to be Jordan if \begin{itemize}\item
   $X$ is projective (Sh. Meng, D.-Q. Zheng, \cite {MengZhang});
 \item $X$ is a compact complex   K\"{a}hler   manifold    (J.H. Kim, \cite{Kim}).
 % $Z$ is a differential manifold with Euler number $e(Z)\ne 0$ (\cite {TurullM}). 
 \end{itemize}
  Moreover, the connected identity component  $\Aut_0(X)\subset \Aut(X)$ of $\Aut(X)$ is Jordan for every compact  complex space $X$
  (\cite[Theorems 5  and 7]{Pop18}).

  Groups  $\Bir (X)$ and $\Bim (X) $ of birational and bimeromorphic transformations, respectively, are    more complicated. 
 
 \begin{Example}\label{dim2}
   In the case of  projective  varieties  $X$, proven was  by  V.L. Popov \cite{Pop} that $\Bir(X)$ is Jordan if $\dim(X)\le 2$ and $X$ is not birational to a product of an elliptic curve and $\BP^1$. (The case of $X=\BP^2$ was done earlier by J.-P. Serre, \cite{Serre1}.)\end{Example}
 
 Consider the following {\bf LIST}  of manifolds: 

\begin{enumerate}\item $E$ - an elliptic curve;\item
 $A_n$ - an abelian variety  of dimension $n;$\item $T:=T_{n,a}$ - a complex torus with dimension $\dim  T=n$ and algebraic dimension $a(T)=a;$\item  $S_b$ - a bielliptic surface;\item  $S_{K1}$ - a surface of Kodaira dimension 1;\item  $S_{K}$ - a  Kodaira  surface (it  is not a   K\"{a}hler   surface).\end{enumerate}
 
 \begin{Example}\label{exbound}  
   \begin{itemize}\item[(1)]
     If $S$ is a projective   surface  with non-negative Kodaira dimension  then $\Bir(S)$ is bounded unless it appears on the {\bf LIST},   \cite[Theorem1.1]{PS19-1}; 
\item[(2)] If $X$ is a non-uniruled projective variety with irregularity $q(X)=0,$ then $\Bir(X)$ is bounded,  \cite[Theorem1.8]{PS14}.\end{itemize}\end{Example}

 \begin{Example}\label{jordan}     
 \begin{itemize} 
 \item[(1)]  $\Bir(X)$    is Jordan for a projective variety $X$  if   either $X$ is  not uniruled   or $X =\BP^N$ (the latter case  was proven in \cite{PS14} modulo the Borisov-Alekseev-Borisov  conjecture that was later established by C. Birkar \cite{Bi});
  \item[(2)]  If $X$ is a uniruled smooth  projective variety that is a non-trivial conic bundle over a non-uniruled smooth projective
  variety $Y$  then  $\Bir(X)$ is Jordan  (\cite{BZ17});
 \item[(3)] 
 If $X$ is a projective threefold then $\Bir(X)$ is Jordan unless $X$ is birational to a direct product $E\times \BP^2$ or $S\times \BP^1,$ where a surface $S$  appears in the {\bf LIST},
 \cite {PS18};
  \item[(4)] If $X$ is a non-algebraic compact uniruled   complex   K\"{a}hler   threefold   then $\Bim(X) $ is Jordan unless  $X$ is either the projectivization of a rank two vector bundle over $T_{2,1}$ (and $a(X)=1$)  or  $X$ is bimeromorphic to $\BP^1\times T_{2,1}$ (and $a(X)=2$),   \cite{PS19-2},\cite{CP2}.\end{itemize}\end{Example}
  
 \begin{Example}\label{nonjordan} \begin{itemize} \item[(1)]  If $X$ is a projective variety, birational to $\BP^m\times A_n, \ n,m>0,$ then  $\Bir(X)$ is not Jordan \cite{Zar14};  
  \item[(2)] The group  $\Bim(X) $ is not Jordan for a certain  class of $\BP^1$-bundles     (including the trivial ones) over complex tori of positive algebraic dimension,  \cite{Zar19}. 
 \end{itemize}\end{Example} 

\begin{comment}

  \begin{Example}\label{vveryjordan}    
   If $X$ is a complex   K\"{a}hler   manifold of non-negative Kodaira dimension, then $\Aut(X)$ is very Jordan (\cite[Proposition 5.11]{Fu78} and \corref{vj}   of this paper)
 \end{Example}
 \end{comment}
 These examples show that the worst case  scenario for Jordan properties of $\Bim(X)$ or $\Bir(X)$  occurs 
 when   $X$ is  a uniruled variety (K\"{a}hler   manifold)     that is  fibered over a torus of positive algebraic dimension.

  The second goal of this paper is to check 
 what happens in the similar situation when a compact complex manifold is uniruled and fibered over a  torus of  algebraic dimension zero. It appears that the Jordan properties are drastically  different from the situation when the  torus  has  positive algebraic dimension.

    Let $X,Y$ be compact connected   complex     manifolds  endowed with a holomorphic map   $p:X\to Y. $ Assume that \begin{itemize}\item
       dimension of  the fiber $P_y:=p^{-1}(y)$ is 1 for every point $y\in Y;$  \item
 there is an analytic subset $Z\varsubsetneq Y$ such that for every point  $y\not\in Z$
the fiber $p^*(y)$ is reduced and isomorphic to $\BP^1.$ \end{itemize}

 In this situation we call  $P_y, y\not\in  Z,$ a general   fiber and $X$ (or a triple  $(X,p,Y) $) an  equidimensional rational bundle over $Y.$  (Such bundles appear naturally in the classification of non-projective smooth compact   K\"{a}hler   uniruled threefolds \cite{CP2}.)  If $X$ is a holomorphically locally trivial fiber bundle over $Y$ with fiber $\BP^1$ we call it a $\BP^1$-bundle.  If $X$ is a projectivization $\BP(E)$   of a rank 2 holomorphic vector bundle $E$ over $Y,$ we will say that $X$ is a linear  $\BP^1$-bundle over $Y.$ 
  We consider manifolds $Y$  with $a(Y)=0$ meeting certain additional conditions.  
  %(see \defnref{poor}  below). We prove that  $\Aut(X)$ and $\Bim(X) $
 %are very Jordan if $X$ is a non-trivial equidimensional $g-$conic bundle over $Y.$ 

 \begin{Definition}\label{poor} We say that a compact connected complex        manifold $Y$ of positive dimension is {\sl  poor} if it enjoys the following properties.
\begin{itemize}
\item
The algebraic dimension $a(Y)$ of $Y$ is $0$.
\item
$Y$ does not contain analytic subspaces of codimension 1.
\item
$Y$  does not contain  rational curves, i.e., it is meromorphically hyperbolic in the sense of Fujiki \cite{Fu80}.
\end{itemize}\end{Definition}

 A  complex torus  $T$ with $\dim(T)\ge 2 $ and  $a(T)=0$ is a { poor}  K\"{a}hler manifold.  Indeed, a complex torus  $T$  is a   K\"{a}hler   manifold that does not contain rational curves.
 If $a(T)=0, $ it contains no analytic subsets of codimension 1 \cite[Corollary 6.4,   Chapter 2]{BL}.  An explicit example of such a torus is given in \cite[ Example 7.4] {BL}. Another 
  example of a  {poor} manifold is   provided by  a non-algebraic $K3$ surface  $S$ with $\NS(S)={0}$   (see \cite[Proposition  3.6, Chapter VIII]{BHPV}).

  \begin{remark}
1) Clearly, the complex dimension of a {poor} manifold is at least 2.  

2)  A {\sl generic} complex torus of given dimension $\ge 2$ has algebraic dimension 0 and therefore is poor.

%3)  See   \cite{CP,CHP,HP} for properties and  examples of poor manifolds of dimension $\ge 3$.
%4) Generic analytic K3 surfaces are poor \cite{MS}.
 \end{remark}

  Let  $(X,p,Y) $  be an  equidimensional rational bundle over  a poor manifold  $Y.$  Since   $Y$
 contains no rational curves,  there are no non-constant holomorphic maps 
$\BP^1\to Y.  $  It follows that  every map $f\in\Bim(X)$ is $p-$fiberwise, i.e., there exists a group homomorphism $\tau:\Bim(X)\to \Aut(Y)$ 
(see \lemref{tau})  such that for all $f\in\Bim(X)$
$$p\circ f=\tau(f)\circ p.$$  We denote by $\Bim (X)_p,$ ($\Aut(X)_p$)  the kernel of $\tau,$ i.e.,  the subgroup of all  those $f\in\Bim(X) $ (respectively, $f\in\Aut(X)$)
that leave every fiber  $P_y:=p^{-1}(y), y\in Y,$  fixed. 
We  prove the following 

\begin{Theorem}\label{IntroMain1}  Let  $(X,p,Y) $  be an equidimensional rational bundle over a  poor  manifold 
 $Y.$    Then:
 \begin{itemize}\item   $(X,p,Y) $ is a $\BP^1-$bundle
 (see  \propref{p1bundle}); \item $\Bim(X)=\Aut(X)$  (see  \corref{bimaut});
\item The restriction homomorphism $\Aut(X)_p\to \Aut(P_y),  \ f\to f |_{P_y}$ is a group embedding.  Here $P_y=p^{-1}(y)$ for any point $y\in Y$  (\lemref{caseA},\lemref{caseB},  and Case  {\bf C}(h)\ of \secref{p1}).
\end{itemize}
Assume additionally that   $Y$ is K\"{a}hler   and $X$ is not bimeromorphic to the direct product    $Y\times \BP^1.$ Then:
\begin{itemize}
\item   The connected  identity component  $\Aut_0(X)$ of  complex Lie  group $\Aut(X)$   is {\bf commutative}  (\thmref{main1}); 
\item Group $\Aut(X)$ is very Jordan. Namely, there is a short exact sequence  
 \begin{equation}\label{goal}     1\to \Aut_0(X)\to \Aut(X)\to F\to 1, \end{equation}
where  $F$ is a bounded   group (\thmref{main1});
\item Commutative group $\Aut_0(X)$ sits in a short exact sequence of complex Lie groups\begin{equation}\label{goal2}
1\to\Gamma\to \Aut_0(X)\to H\to 1,\end{equation}
where $H$ is a complex  torus and  one of the following  conditions holds  (\propref{ap} and Equation \eqref{gamma300}):
 \begin{itemize}\item $\Gamma=\{id\}$, the trivial group;
 \item $\Gamma\cong\BC^{+},$  the additive group of complex numbers;
\item $\Gamma\cong\BC^{*},$   the multiplicative  group of complex numbers.  \end{itemize}
%\item In particular, 
%$\Bim(X)$ is strongly Jordan.
\end{itemize}\end{Theorem}

 The paper is organized as follows. Section \ref{prel} contains preliminary results about automorphisms of equidimensional rational bundles and meromorphic groups 
 in a sense of A. Fujiki. In Section \ref{poormanifolds}  we deal with equidimensional  rational 
 bundles over poor manifolds and prove that every such equidimensional rational bundle is a $\BP^1-$bundle. 
 In Section \ref{p1}, we study  $\BP^1$-bundles $X$  over   a poor manifold $T$ and classify their nontrivial fiberwise   automorphisms
 in terms of the   corresponding fixed points sets. In particular, we prove that $\Bim(X)=\Aut(X).$  
  In Section \ref{abelian}  assuming that our poor manifold $T$ is K\"{a}hler  we  prove that the connected identity component $\Aut_0(X)$ of $\Aut(X)$  is  commutative. In Section  \ref{example} we provide a class of examples of $\BP^1$-bundles $X$ over complex tori $T$ of 
  algebraic dimension 0 that do not admit a section but admit a bisection that coincides with the set of fixed points of a certain equivariant  automorphism.
 
 {\bf Acknowledgements}. We are grateful to Frederic Campana,  Vladimir L. Popov, and Constantin Shramov for useful, stimulating discussions and very helpful  comments.  We thank Igor Dolgachev for help  with references. We are deeply grateful to the  referees:  their valuable comments lead us to substantial revision of the paper.

\section{Preliminaries and Notation}\label{prel}

We assume that all complex manifolds under  consideration are connected and compact. 
 We use the following  notation  and assumptions.

 {\bf Notation  and Assumptions 1.}
 
  \begin{itemize}
%\item The ground field is $\BC.$
\item  $\Bim(X), \Aut(X)$ stand for the groups of  all   bimeromorphisms and all  automorphisms of  a  complex  manifold $X,$    respectively.
\item  $\Aut_0(X)$  stands for the connected identity  component of  $ \Aut(X)$  (as a complex Lie group).
\item If $p:X\to Y$ is a morphism of complex manifolds, then $\Bim(X)_p, \Aut(X)_p$ is the subgroup of all $f\in \Bim(X) $ (respectively, $ f\in \Aut(X)$)  such that $p\circ f=p.$
\item $\cong$ stands for  `` isomorphic groups" (or isomorphic complex Lie groups if the groups involved are the ones),  and $\sim$ for biholomorphically  isomorphic complex manifolds.
\item   $id$  stands for identity automorphism.
%\item $\BC^2_{x,y}$ stands for $\BC^2$ with coordinates $(x,y);$
\item $\BP^n_{(x_0:...:x_n)}$ stands for a complex projective space $\BP^n$ with homogeneous coordinates $(x_0:...:x_n).$
\item  $\BC_z, \ov{\BC}_z\sim\BP^1$  is the  complex line (extended complex line, respectively) with coordinate $z$.
\item    For an element $m\in\PGL(2,\BC)$  we define  $\mathrm{DT}(m):=\frac{\tr^2(M)}{\det(M)}$ where $M\in\GL(2,\BC)$ is any matrix representing $m,$  \ $\tr(M)$ and $\det(M)$ are the trace and the determinant of  $M,$ respectively.   $\mathrm{DT}(m)=4$ if and only if $m$ is proportional either to  the identity matrix or to a unipotent matrix.
\item $\BC^+$ and $\BC^*$ stand for   complex Lie groups $\BC$ and $\BC^*$ with additive and multiplicative group structure, respectively.
\item $\dim (X),  \ a(X) $ are  the  complex and algebraic dimensions of a compact complex manifold $X$, respectively.
\item    Let $X,Y$ be two compact connected irreducible reduced analytic complex spaces. A meromorphic map   $f:X\to Y$ relates to every point  $x\in X$  a subset $f(x)\subset Y$
(the image of $x$)  such that the  following conditions are met  \begin{enumerate}\item The graph $G_f:=\{(x,y) \ | y\in f(x)\subset X\times Y\} $ is a connected irreducible
complex analytic subspace of $X\times Y$ with $\dim (G_f)=\dim ( X);$
\item There exist an open dense subset $X_0\subset X$ such, that $f(x)$ consists in one point for every $x\in X.$ \end{enumerate} 
\item We say that a compact complex manifold $Y$ contains no rational curves if there are no {\bf nonconstant} holomorphic maps $\BP^1 \to Y$.
%there is no analytic subset in $Y$ that is biholomorphic to $\BP^1;$ 
\item Following  A. Fujiki, we call a compact complex manifold  {\bf meromorphically hyperbolic} if it contains no   rational curves, (\cite{Fu80}).
\item   According  to  A. Fujiki \cite[Definition 2.1]{Fu78},   a  {\bf meromorphic structure}
on a complex Lie group $G$ is  a compactification $G^*$ of $G$  such  that 

%\begin{enumerate}\item  
{\sl the group multiplication $\mu: G\times    G\to G$ extends to a meromorphic map   $\mu^*: G^*\times   G^*\to G^*$   and  $\mu^*$ is holomorphic on $G^*\times   G\cup  G\times   G^*$}.%\end{enumerate}

\item    Following  A. Fujiki   we say that a complex Lie group $G$ acts meromorphically    on a complex manifold $Z$ if 

\begin{enumerate}\item G acts biholomorphically on $Z;$ 
\item  there is a  meromorphic structure $ G^*$ on $G$  such that the   $G-$action $\sigma:G\times Z\to Z $ extends  to a meromorphic map $\sigma^*:G^*\times Z\to Z$ (see \cite[Definition 2.1]{Fu78}  for details).
\end{enumerate} 
\end{itemize}

 It was proven in \cite{Fu80}  that if   a      manifold   $Y$  is meromorphically hyperbolic  then 
 
 \begin{enumerate}\item
 every meromorphic map $f:X\to Y$  is holomorphic for any complex manifold $X.$
 \item If, in addition,  $Y$ is K\"{a}hler then \begin{itemize}\item every   connected  component  of the set $H(Y,Y)$ of all holomorphic maps $Y\to Y$   (regarded as a certain subspace of the Douady 
 complex analytic space $D_{Y\times Y}$) 
 is compact;
 \item  in particular,  $\Aut_0(Y)$ is a compact  complex   Lie group, that is isomorphic to a certain   complex torus $\Tor(Y)$ (see also \cite[Corollary  3.7]{Fu78}). 
\item actually,  $\Aut_0(Y)$  is isomorphic to a complex torus for any compact complex   K\"{a}hler   manifold $X$ of non-negative Kodaira dimension \cite[Proposition 5.11]{Fu78}). \end{itemize}
  \end{enumerate}

 In general,  let $Z$ be a  compact complex  connected  K\"{a}hler    manifold. The group  $\Aut_0(Z)$ acts meromorphically on $Z$, and 
  the analogue of the Chevalley  decomposition for algebraic groups is valid for complex Lie group $\Aut_0(Z):$
\begin{equation}\label{Chevalley}
1\to L(Z)\to \Aut_0(Z)\to \Tor(Z)\to 1\end{equation}
where $L(Z) $ is     bimeromorphically     isomorphic to a linear group, and $\Tor(Z)$ is a complex  torus        (\cite[Theorem 5.5]{Fu78},     \cite [Theorem 3.12]{Lie}, \cite[Theorem 3.28]{CP}).

If $L(Z)$  in \eqref{Chevalley} is not trivial,   $Z$ contains a rational curve.  Moreover, according  to
\cite[Corollary  5.10]{Fu78},  $Z$ is bimeromorphic to a fiber space  whose general fiber is $\BP^1.$

The next Proposition is similar to Lemma 3.1 of \cite{Kim}.
\begin{Proposition}\label{abounded}
Let $X$ be a connected complex compact   K\"{a}hler   manifold  and   $F=\Aut(X)/\Aut_0(X).$ Then group 
$F$ is bounded.\end{Proposition}
\begin{proof}  By functoriality, there is the  natural      group homomorphism 
\begin{equation}\label{phi}
\phi:F\to \Aut(H^2(X,\BQ)), \ f\in \Aut(X)\to f^*\in \Aut_{\BQ}(H^2(X,\BQ)).\end{equation}
Connected Lie group  $\Aut_0(X)$ is arcwise connected. Hence, $f^*$ is the identity map  for all $f\in \Aut_0(X).$
  The image $\phi(\Aut(X))$ is bounded, since it is a subgroup of a bounded (thanks to Minkowski's theorem,  \cite[Section 9.1]{Serre})  group $\Aut_{\BQ}(H^{2}(X,\BQ))\cong \GL(b_2(X),\BQ).$  (Here $b_2(X)=\dim_{\BQ}H^{2}(X,\BQ)$ is the second Betti number of $X$). On the other hand,
if $f\in \ker (\phi) $   then  its action on $H^{2}(X,\BR)=H^{2}(X,\BQ)\otimes_{\BQ}\BR$ is trivial as well.  Thus, if $\omega$ 
is a   K\"{a}hler   form on $X, $ and   $\ov{\omega} $ is its cohomology class in  $H^{2}(X,\BR),$  and  if $f\in \ker (\phi), $   then 
\begin{equation}\label{omega}f^*(\ov \omega)=\ov \omega.\end{equation}
  Let  
$\Aut(X)_{\ov \omega}\subset \Aut(X)$ be the  subgroup  of all  automorphisms meeting condition \eqref{omega}.   We have   $ \Aut_0(X) \subset \ker(\phi)\subset \Aut(X)_{\ov \omega}. $  Since  the quotient group  $\Aut(X)_{\ov \omega}/ \Aut_0(X)$ is finite (\cite[Theorem     4.8]{Fu78}, 
     \cite[Proposition 2.2 ]{Lie}), the  quotient  $ \ker (\phi)/\Aut_0(X)\subset \Aut(X)_{\ov \omega}/\Aut_0(X)$ is a finite group. 
Thus, we have a short  exact sequence   of groups:
$$1\to \ker (\phi)/\Aut_0(X)\to(\Aut(X)/\Aut_0(X)=F)\to  \phi(\Aut(X))\to 1.$$

The  group $\ker (\phi)/\Aut_0(X)$   is finite, the  group  $\phi(\Aut(X))$     is bounded, thus $\Aut(X)/\Aut_0(X)$ is also bounded.\end{proof}

 \begin{remark}  Our  proof  was   inspired by the proofs of Lemma 3.1 of  \cite{Kim}, and  Lemma 2.5 of              
\cite {MengZhang}. Namely, Lemma 3.1 of Jin Hong Kim,  \cite{Kim}, states the following.
 
 Let $X$ be a normal compact K\"{a}hler variety. Then there exists a positive
integer $l,$ depending only on X, such that for any finite subgroup $G$ of $\Aut(X)$
acting biholomorphically and meromorphically on $X$ we have
$[G : G \cap \Aut_0(X)] \le l.$
 
 We cannot use straightforwardly  this Lemma since a finite subgroup of $\Aut(X)/\Aut_0(X)$ may not be isomorphic to a quotient $G/(G\cap  \Aut_0(X))$ where $G$ is a finite subgroup  of $\Aut(X).$ 
 \end{remark}
 
 \begin{corollary}\label{vj}  Let $X$ be a compact complex K\"{a}hler manifold of Kodaira dimension $\varkappa(X)\ge 0.$ Then $\Aut(X)$ is very Jordan. \end{corollary}
 \begin{proof} In view of \propref{abounded} it is sufficient to prove that $\Aut_0(X)$ is commutative.   But this follows from \cite[Proposition 5.11]{Fu78}  that  asserts that 
 $\Aut_0(X)$ in this case  is a complex torus.\end{proof}. 
 
 \section{Equidimensional rational bundles over poor   manifolds}\label{poormanifolds}

We will use the following property of poor manifolds.
\begin{Lemma}\label{cover} Let $X,Y$ be connected compact manifolds, and let $f:X\to Y$ be a unramified  finite holomorphic covering.   Then 
\begin{itemize}\item If $Y$ is  K\"{a}hler, so is  $X;$   \item If $Y$ contains no rational curves, so does $X;$
\item If $Y$ contains no analytic subsets of  $\codim  \ 1$, so does $X;$
\item  If $Y$ is poor, so is   $X.$\end{itemize}
\end{Lemma}
\begin{proof}
Indeed, \begin{enumerate}\item  If $\omega$ is a   K\"{a}hler   form on $Y,$ then  its pullback  $f^*\omega$ is a   K\"{a}hler   form on $X$, thus $X$ is a   K\"{a}hler   manifold. 
\item If $X$ contained   a rational curve  $C$ then  $f(C)$ would be a rational curve in $Y$.
\item If $X$ contained  an   analytic subset $Z$ of $\codim \  1,$ then   $f(Z)$ would be a $\codim  \ 1$ analytic subset in $Y.$
\item  Thus if  $Y$ is poor, according (1) and (2), $X$  contains neither rational curves nor analytic subsets of $\codim \  1.$  In particular,  $a(X)=0.$  Thus, $X$ is poor. 
 \end{enumerate}\end{proof}

 An equidimensional rational bundle $(X,p,Y)$ defines  the holomorphically locally trivial    fiber bundle  with  fiber  $\BP^1$  over a certain  open dense subset $U\subset Y. $ 
Indeed,  by definition,  there is an open dense subset $U\subset Y$   of points $y\in Y$ such that     for    all $y\in U$ 
  the  fiber $P_y=p^{-1}(y)\sim \BP^1.$    By a theorem of W. Fischer and  H. Grauert  (\cite{FG}), the triple  $ (p^{-1}(U),p,U)$ is 
  a holomorphically locally trivial fiber   bundle.  Actually, we may (and will) assume that $U$ is a complement of an analytic subset of $Y,$   since the image of the set of  points  where $p$ is singular  is an analytic subset  (see, for example, \cite[Theorem  1.22]{PR}).

\begin{Definition}\label{fiberwise} (Cf. \cite{BZ18}).  Let $X, Y , Z$ be three  complex manifolds,  $f:X\to Y, g:Z\to Y$ be  holomorphic maps,  and $ h:X\to Z$ be a meromorphic map.  We say that 
$h$ is $f,g$-fiberwise  if there exists a holomorphic map $\tau(h):Y\to Y$  that may be included into the following commutative diagram.

\begin{equation}\label{diagram1001}
\begin{CD}
X @>{h}>>Z\\
@V f VV @Vg VV \\
Y @>{\tau(h)}>> Y
\end{CD}.
\end{equation}
 If $X=Z$ and $f=g$ we say that $h$ is $f-$fiberwise  (or equivariant).  If $\tau(h)=id$ we  say that $h$ is fiberwise.
 \end{Definition}

\begin{Lemma}\label{tau} (Cf.  \cite[Lemma 3.4]{BZ17}  for algebraic case). 
Assume that   $Y$ is a  connected compact complex    meromorphically hyperbolic   manifold and let $(X,p_X,Y)$ and $(Z,p_Z,Y)$   be   two equidimensional rational  bundles
over $Y.$

Then 
any surjective   meromorphic map $f:X\dashrightarrow Z$  is $p_X,p_Z-$ fiberwise. 
 \end{Lemma}

\begin{proof}   Consider the  meromorphic map $g_f:=p_Z\circ f :X\to Y.$  It is holomorphic (\cite[Proposition 1]{Fu80})  since $Y$ has no rational curves. Let $U\subset Y $ be  a dense  Zariski open subset of $Y$ such   that $ (p_X^{-1}(U),p_X,U)$   is a
 holomorphically locally trivial fiber   bundle. Take a  fiber $P_u=p_X^{-1}(u), u\in U.$  Since $g_f$ is holomorphic, image $g_f(P_u)$ may be either   a point or a rational curve. Since $Y$ contains no rational curves,  the restriction  map
\begin{equation}\label{constant}g_f \bigm |_{P_u} : \ P_u\to Y \ \text{is a constant map}. \end{equation}   Since  $U$ is dense  and the set of points  $y\in Y$ such that $g_f \bigm |_{P_y}$ is a constant, is closed, we get that  $g_f \bigm |_{P_y}$ is constant  for any $y\in Y.$ Put $\tau(f)(y):=g_f \bigm |_{P_y}.$

For  a  fiber  $P_u$ with   $u\in U$, there exists an open  neighborhood 
$W$ of  $u$ in $U$  such that $V=p_X^{-1}(W)$ is $p_X, p_1-$fiberwise isomorphic to $ W\times\BP^1_{(x:y)},$ where
$p_1$ stands for the natural projection to the second factor.
 %Let $z=x/y.$
  Then  for $w\in W$ we have $\tau(f)(w)=p_Z\circ f(w, (0:1) ),$ hence is a
 holomorphic function on $w.$  Thus,   $\tau(f)$
is holomorphic on $U,$    defined and continuous on $Y.$

Let $y\in Y\setminus U$  and $z=\tau(f)(y)$.   Let us choose open neighbourhoods  $W_y, W_z,\subset Y$ of $y,z$ respectively such that 
\begin{itemize}
\item
both $W_y$ and $ W_z$   are biholomorphic to an open ball in $\BC^n$  with induced coordinates 
$y_1,\dots,y_n$  and $z_1,\dots,z_n$  respectively;
\item
$\tau(f)(W_y)\subset W_z.$
\end{itemize}

Then the induced functions $\tau(f)^*(z_i)$ are holomorphic in $W_y\cap U,$ defined      and locally bounded   in $W_y$ thus, by the first Riemann continuation Theorem   (\cite[Chapter 1, C, 3] {GR}, \cite[Section 2.23]{Fi}) are holomorphic in $W_y.$  Hence, $\tau(f)$ is a holomorphic map. \end{proof}

\begin{Corollary}\label{tauh} For an equidimensional  rational   bundle $(X,p,Y)$ over a meromorphically hyperbolic (complex connected compact) manifold $Y$ there is a  group homomorphism  $\tau:\Bim(X)\to \Aut(Y)$  such that $$p\circ f=\tau(f)\circ p$$ for every $f\in\Bim(X).$
Thus every $f\in\Bim(X)$ is $p-$fiberwise.
\end{Corollary}
\begin{remark}\label{Fu} If, in addition,  $Y$ is K\"{a}hler, then  $\tau(\Aut_0(X))$ has the natural meromorphic structure and 
the group homomorphism $$\tau \bigm |_{\Aut_0(X)}:\Aut_0(X)\to\tau(\Aut_0(X))$$   is a meromorphic map, in particular, 
$\tau$ is 
a holomorphic homomorphism of complex Lie groups and  $\tau(\Aut_0(X))$ is a complex Lie subgroup of $\Aut(Y)$(A. Fujiki, \cite[Lemma 2.4, 3)]{Fu78}).\end{remark}

\begin{Proposition}\label{p1bundle}  Let $(X,p,Y)$ be an equidimensional rational   bundle. Assume that $Y$ contains no
analytic subsets  of codimension   1.    Then $(X,p,Y)$  is a $\BP^1-$bundle.\end{Proposition}

\begin{proof}  Let   $\dim  Y=n, $  and $$S=\{x\in X   \ | \  \rk (dp)(x)<n\}$$  be the set of all points in $X$
where the differential $dp$  of $p$ does not have the maximal rank. Then $S$ and $\tilde S=p(S)$ are analytic subsets of $X$
and  $ Y,$ respectively  (see, for instance,
\cite [Theorem 2, Chapter  VII]{Narasimhan},
 \cite[Therem  1.22]{PR}, \cite{Re}).  Moreover, $\codim(\tilde S)=1 $ (\cite{Ra}).  Since 
$Y$ contains no
analytic subsets of codimension  \ 1, we obtain: $\tilde S=\emptyset.$  Thus the holomorphic map $p$ has no singular fibers. \end{proof}

\begin{remark}\label{compana}We used the  following theorem of C.P. Ramanujam (\cite{Ra}).   Let $X$ and $Y$ be connected complex manifolds, $f:X\to Y$ a proper flat holomorphic map  such that the general fiber is Riemann sphere, $D$ the set of point in $X$ such that $df$ is not of maximal rank  and $E=f(D).$  Then $E$ is pure of codimension 1 in $Y.$

In the algebraic case this result was proven by Igor  Dolgachev, \cite{Dol}.  \end{remark}

  Consider now  a  $ \BP^1$-bundle  over  a compact complex   connected   manifold $Y,$  i.e.,  a  triple $(X,p,Y) $ such that 
  $X$ is  a   holomorphically locally trivial fiber bundle over $Y$ with  fiber  $\BP^1$  and 
with the  corresponding projection $p:X\to Y.$ 
 Let us fix some notation.

{\bf Notation  and Assumptions 2.}

\begin{itemize}
\item   $P_y$ stands for the fiber $p^{-1}(y).$
\item We call the  covering $Y=\cup U_i$   by open subsets of $Y$  {\it fine} if  for every $i$      there exist an isomorphism  $ \phi_i$ of  $ V_i=\pi^{-1}(U_i)$ to direct product $ U_i$ and $\BP^1_{(x_i:y_i)}$ that is compatible with the natural projection $\mathrm  {pr}$ on the second factor (i.e., $p,\mathrm  {pr}-$fiberwise).
  In other words $ V_i\subset X$ stands for $p^{-1}(U_i):$ we have an induced isomorphism $\phi_i:V_i\to  U_i\times \BP^1_{(x_i:y_i)} $  and $(y,(x_i:y_i))$  are coordinates in $V_i;$
 
 In $(U_i\cap  U_j)\times\BP^1_{(x_i:y_i)}$
defined is a  holomorphic  map $\Phi_{i,j}=(id, A_{i,j}(y)):$
$$(y,(x_i:y_i))\to(y, (x_j:y_j))$$ such that 
\begin{itemize}
\item $A_{i,j}\in\PGL(2,\BC)$  with representative $$\tilde  A_{i,j}(t)=\begin{bmatrix}a_{i,j}(t) &b_{i,j}(t)\\c_{i,j}(t) &d_{i,j}(t)\end{bmatrix}\in  \GL(2,\BC);$$
\item $ (x_j:y_j)=A_{i,j}(y)((x_i:y_i))=((a_{i,j}(y)x_i +b_{i,j}(y)y_i):(c_{i,j}(y)x_i+d_{i,j}(y)y_i));$
\item   $A_{i,j}(y)= A_{j,i}(y)^{-1};$
\item the following    diagram commutes:
\begin{equation}\label{diagram1}
\begin{CD}
V_i  \cap  V_j @>{ \phi_i}>> (U_i\cap  U_j)\times \BP^1_{x_i:y_i}  \\
@V id VV @V\Phi_{i,j} VV \\
 V_i  \cap  V_j @>{ \phi_j}>> (U_i\cap  U_j)\times \BP^1_{x_j,y_j} 
\end{CD}.
\end{equation}.
\item  $A_{i,j}(y) $ depend holomorphically   on $y$ in  $U_i\cap  U_j;$ 
%Note that $A_{i,j}(t)$ are defined up to multiplication by scalar.\end{itemize}
\item $A_{i,j}(y) A_{j,k}(y) = A_{i,k}(y).$
\end{itemize}\end{itemize}

 \begin{Lemma}\label{extension}  Let $Z\subset Y$  be an analytic subset of $Y$  with $\codim Z\ge 2, $ and  $\tilde Z:=p^{-1}(Z)\subset X.$ Let  $ f\in \Bim (X)$   be $p-$fiberwise.    If  $f$ is defined at every point $x\in X \setminus\tilde Z$  
 then $f\in \Aut(X).$ 
\end{Lemma}

\begin{proof}  

Let   $ \{U_i\}$ be a fine covering of $Y$.    Since $f^{-1}\in\Bim (X)$ is $p-$fiberwise as well, $g:=\tau(f)$ 
is a biholomorphic map. 
  Let $z \in \tilde Z$ and $W\subset U_i$  be an  open neighborhood of $p(z)$ such that $g(W)\subset U_j$ for some $j.$ Let  $B:=W\cap Z,   A:=W\setminus B.$  For every  $t\in   A $ the restriction $f\bigm  |_{P_t}$ is  an isomorphism  of  $P_t$ with $P_{g(t)}$  defined in corresponding coordinates by an element of $\PSL(2,\BC). $  Thus, we have a holomorphic map 
  $$\psi_{W,f}:A\to \PSL(2,\BC)$$  such that 
  $$ f(t, (x_i:y_i))=(g(t),\psi_{W,f}(t)((x_i:y_i))). $$
       Since $\PSL(2,\BC) $  is an affine variety,   and $\codim B\ge 2,$   by  the  Levi's continuation Theorem
(\cite{Levi}, see also \cite[Theorem 4 Chapter   VII]{Narasimhan} or \cite[Section 4.8]{Fi})
  there exists a holomorphic extension $\tilde \psi_{W,f}:W\to \PSL(2,\BC).$ 
  We define $$\tilde f(t, (x_i:y_i))=(g(t),\tilde\psi_{W,f}(t)((x_i:y_i))) \quad  in \quad p^{-1}(W). $$  Thus we can extend $f$ holomorphically at  any point $z\in\tilde Z.$

  Since  outside $\tilde Z$  all the extensions coincide, 
  this global extension is uniquely defined. \end{proof}

\begin{Definition}\label{sectionn} 
 A $n-$section $S$ of $p$ is a  codimension 1  analytic subset $D\subset X$  such that the intersection $X\cap P_y$  is finite for every $y\in Y$ and  consists in $n$ distinct  point for the general $y\in Y.$  A bisection is a $2-$section.
 A section  $S$  of $p$ is a $1$-section.\end{Definition}

\begin{remark}\label{s1}
  For a section $S$ of $p$   there is a holomorphic map  $\sigma:Y\to X$  such that  section  $S=\sigma(Y)$ and $p\circ\sigma=id$  on $Y.$   In every $U_i$ the map  $\sigma $  is defined  by a function $\sigma_i:U_i\to V_i$ such that 
\begin{equation}\label{sigma}
 A_{i,j}(t)\circ \sigma_i(y)=\sigma_j(t)
\end{equation}
for $t\in U_i\cap  U_j.$ \end{remark}

\begin{Lemma}\label{section}  If $Y$ contains no analytic subsets of codimension 1, then \begin{enumerate}\item
any two distinct   sections  of $p$ in $X$ are disjoint;\item a $n-$section has no ramification  points (i.e the intersection $X\cap P_y$  consists in $n$ distinct  point for every $y\in Y$)\end{enumerate} \end{Lemma}
\begin{proof}

(1)  If  a section $S=\sigma(Y)$   meets  a  section $R=\rho (Y)$  then  the intersection $S\cap R$ is either empty or has codimension $2$ in $X.$
Since none of sections contains a fiber,  $p(S\cap R)$  is either empty or has codimension 1 in $Y.$
Since $Y$  carries no  analytic subsets of codimension 1,  
 $p(S\cap R)=\emptyset.$ 

(2) Let $R$ be an $n-$section   of $p,$ let $A$ be the set of all points  $x\in R$  where the restriction $p \bigm  |_R:R\to Y$ of $p$ onto $R$ is not locally biholomorphic.  Then 
the image   $p(A)$  is either empty or  has pure codimension 1  in $Y$(\cite[Section 1,  9]{DG}, \cite[Theorem1.6]{Pe}, \cite{Re}). 
Since $Y$  carries no  analytic subsets of codimension 1,  $p(A) =\emptyset.$ Hence, $A=\emptyset.$
\end{proof}

 \section {$ \BP^1$-bundles  over poor manifolds.}\label{p1}

We now fix a poor complex  manifold $T$ and 
 consider  a  $ \BP^1-$bundle  over $T,$  i.e., a triple $(X,p,T)$ such that 

\begin{itemize}\item $X$  and  $T$ are connected compact complex manifolds;
\item  $T $  contains neither  a  rational curve  nor  an analytic subspace  of $\codim \  1,$ and   algebraic dimension $a(T)=0;$
\item  $X$ is  a   holomorphically locally trivial fiber-bundle over $T$ with  fiber  $\BP^1$  and 
with the  corresponding projection  map  $p:X\to T.$ \end{itemize}

\begin{Corollary}\label{bimaut} $\Bim(X)=\Aut(X).$\end{Corollary}
\begin{proof}
  Since $T$ contains no rational curves, by   \lemref{tau}  every $f\in\Bim(X)$ is $p-$fiberwise.  For $f\in \Bim(X)$ let $\tilde S_f$ be the  indeterminancy locus of $f$  that is an analytic subspace of $X$ of codimension at least 2 (\cite[page 369]{Re}). 
  Let   $S_f={p(\tilde S_f)},$ which is an analytic subset of $Y$ (\cite{Re}, \cite [Theorem 2, Chapter  VII]{Narasimhan},).
Since $T$ 
contains  no analytic subsets of codimension 1, $\codim  S_f\ge 2.$  Moreover, $f$ is defined at   all points   of  $X\setminus p^{-1}(S_f).$ 
 By \lemref{extension} both  $f\in \Bim(X)$  and  $f^{-1}\in \Bim(X)$ may be holomorphically extended to  $X,$   hence we get 
$ \Bim(X)=\Aut(X).$\end{proof}

Recall that by $\Aut(X)_p$ we denote the kernel of   group  homomorphism $\tau:\Bim(X)=\Aut(X)\to \Aut(T),$  i.e., the subgroup of automorphisms   of $X$ that leave every fiber of $p$ invariant. 

   Let $f\in\Aut(X)_p, f\ne id.$  By Lemma 3.1 of   \cite{kostya} we know that the set of fixed points of $f$ is a divisor  in $X.$  The following consideration shows that this divisor  is either a smooth section of $p,$ or a union of two disjoint sections of $p,$ or a smooth $2-$section.

\begin{Proposition}\label{fixedpoints}  Assume that $X\not\sim T\times\BP^1.$  Let $ f\in\Aut(X)_p, f\ne id,$ and 
 let $S$ be the set of all  fixed points  of $f.$
Then one of  three following cases  holds :\begin{itemize}\item[{\bf  A.}] $S=S_1\cup S_2 $  is a union of two  disjoint sections  $S_1$ and $S_2$ of $p;$
\item[{\bf  B.}] $S$ is a section of $p;$
\item[{\bf  C.}] $S$ is a  $2-$section of $p$ (meeting every fiber at two  distinct 
 points).\end{itemize} 
 \begin{comment}
 2)  In case {\bf A}  the $\BP^1-$bundle $(X,p,T)$ is linear : $X=\BP(E)$, where $E$ is a decomposable rank 2 vector bundle on $T.$  There are exactly two sections  of $p$,   $S_1$ and $S_2.$  Moreover,  $X\setminus S_2$  is a total body of the line bundle $\LL_f$   and $S_1 $ is its  zero section.   
 
 3)  In case {\bf B}  the $\BP^1-$bundle $(X,p,T)$ is linear: $X=\BP(E)$, where $E$ is a indecomposable rank 2 vector bundle on
$T$ and $S$ is its only section. 

4)   In case {\bf  C}  a certain double cover of $X$ is a linear $\BP^1-$bundle  having at least two sections.

5)   The  automorphism $f\in\Aut(X)_p$  is uniquely determined by its restriction to any fiber $P_t$ with $t\in T$  (cf. Lemma 4.3 of 
\cite{kostya}).
 \end{comment}
 \end{Proposition}

\begin{proof}
Let $\{U_i\}$ be a fine covering of $T.$
 Let $id\ne f\in  \Aut(X)_p$  be defined (see { Notation and assumptions 2})   in $V_i$  with coordinates $(t,(x_i:y_i))$ by  $f_i(t,(x_i:y_i))=(t,F_i(t)(x_i:y_i)),$
where \begin{enumerate}
\item
$F_i(t)(x_i:y_i)=(f_{i,11}(t) x_i+f_{i,12}(t)y_i):f_{i,21}(t) x_i+f_{i,22}(t)y_i);$ 
\item  $\tilde F_i(t)=\begin{bmatrix}f_{i,11}(t) &f_{i,12}(t)\\f_{i,21}(t) &f_{i,22}(t)\end{bmatrix}$ 
represents  $F_i(t):=\psi_{U_i,f}(t)\in\PGL(2,\BC)$  (see proof of \lemref{extension}).

\item The set of fixed points of $F_i(t)$ is  the analytic subset of $X$ 
defined by equation\begin{equation}\label{fixed1}
 (f_{i,11}(t) x_i+f_{i,12}(t)y_i):(f_{i,21}(t) x_i+f_{i,22}(t)y_i)=(x_i:y_i),\end{equation} that is 
 
 \begin{equation}\label{fixed2}
               f_{i,12}(t)y_i^2+(f_{i,11}(t)-f_{i,22}(t))x_iy_i-f_{i,21}(t)x_i^2=0.\end{equation}
\end{enumerate}
 It is obviously an analytic subset of $X.$
 In every $U_i$ the  function $$\TD_i(t)=\TD(F_i(t))=\frac{\tr (\tilde F_i(t))^2}{\det (\tilde F_i)}
$$ is defined and 
holomorphic.   Since $F_i(t)$ represent globally defined map $f\in\Aut(X)_p,$
we get  (see {Notation and Assumptions 2})
\begin{equation}\label{13}F_j(t)\circ  A_{i,j}=A_{i,j}\circ F_i(t),
\end{equation} which means that 
\begin{equation}\label{14}\tilde F_j(t)\tilde  A_{i,j}=\lambda_{i,j}(t) \tilde A_{i,j}\tilde F_i(t),
\end{equation}
where $\lambda_{i,j}(t)\ne 0$ are some complex functions  in $U_i\cap U_j.$
From \eqref{14}  we have  \begin{enumerate}
\item
$\TD(t):= \TD_i(t)$ for $t\in U_i, $ is holomorphic  and globally defined on $T,$ hence constant, we denote this number  $\TD_f;$
\item  If $\delta_f=\TD_f-4\ne 0,$ then    fix a square root $A_f:=\sqrt{\TD_f-4}\in \BC^*$  and define $\lambda_f=\frac{T_f+A_f}{T_f-A_f}$
as   the ratio of the eigenvalues of $\tilde  F_i(t)$ (it does  not depend on $i$).  Then for every $i$  one can define   coordinates $(t,u_i), \   u_i\in \ov \BC,$
in $V_i=p^{-1}(U_i)$  in such a way that $f(t, u_i)=(t,\lambda_f  u_i).$       The set $S\cap V_i$
of fixed points of $f$  in $V_i$  is  $\{u_i=0\}\cup\{u_i=\infty\}.$ Thus $S$  is an unramified double cover of $T:$  it may be either  a  union of two disjoint sections or one bisection  (see cases {\bf  A, C} below for details).
\item
If $\delta_f=\TD_f-4=0$ then $\tilde  F_i(t)$ is proportional to a unipotent matrix and for every $i$   one can define
in $V_i=p^{-1}(U_i)$   coordinates $(t,w_i) , \  w_i\in\ov \BC,$ in such a way that $f(t,  w_i)=(t, w_i+a_i(t))$  where $a_i(t)$ are  holomorphic  functions  in $U_i.$ The set $S$ of fixed points  in $V_i$ is thus the union of the section  $\{w_i=\infty\}$ and 
$p^{-1}(R_f),$ where 
$$R_f=\cup\{a_i(t)=0\}=\{t\in T \   :    \   f \bigm |_{P_t}=id\}.$$   Since it has codiemsion 1, it has to be empty 
 (see case {\bf  B}   below for details).\end{enumerate}

 In other words,     for every $t\in T$ the selfmap $F_i(t)$ of  $P_t$ is either identity  map, or has two  fixed points, or has one fixed point.
%Let $(a_i:b_i)_1$ and $(a_i:b_i)_2$  be solutions for \eqref{fixed}
If $\TD_f-4\ne 0$ then  \eqref{fixed2} defines a smooth analytic subset  $S$  of $X$ and   $p^{-1}(t)\cap S$ contains precisely 2 distinct points  for any $t\in T.$
Therefore, $S$ is either an unramified smooth double cover of $T $ or a union of two smooth disjoint sections of $p.$

If $\TD_f-4=0$ then  \eqref {fixed2} defines a smooth section of $p$ over the  complement to an analytic subset $R_f$   of $T$ (that has to be empty) or holds identically on $X.$\end{proof}

Thus we have   the following three cases. 

 {\bf Case A.}  The  set of all fixed points of a non-identity map $ f\in\Aut(X)_p$  is the union of two disjoint sections $S_1$ and $S_2$ of  $p.$ We will say that $f$ has type {\bf A} with   Data $ (S_1,S_2)$ (an ordered pair).
Changing   Data $ (S_1,S_2)$ to Data $ (S_2,S_1)$ would lead to changing $\lambda_f$ for $\frac{1}{\lambda_f}.$ 
 
 \begin{Lemma}\label{caseA} Assume that   $ f \in\Aut(X_p), \ f\ne id,$ has type {\bf A} with   Data $ (S_1,S_2)$ and $X\not\sim T\times\BP^1.$  Then \begin{itemize}
 \item  $X\setminus S_2$ is the total body of a holomorphic line bundle $\LL_f$ with zero section $S_1;$ 
 \item $\LL_f$ has no  other sections;
 \item $\Aut(X)_p$  contains a subgroup $\Gamma_A\cong \BC^*$ of all $ g\in \Aut(X)_p$  with the same Data $ (S_1,S_2)$;
 \item any  automorphism $ g\in \Aut(X)_p$ of type {\bf A} belongs to $\Gamma_A;$
 \item an  automorphism $g\in\Gamma_A$ is uniquely determined by its restriction to any fiber $P_t$ with $t\in T$  (cf. Lemma 4.3 of 
\cite{kostya}).\end{itemize}
\end{Lemma}
\begin{proof}  Similarly to  proof   of \propref{fixedpoints},  let  $\{U_i\}$ be a fine covering of $T.$
 Let $ f\in  \Aut(X)_p, $  be defined    in $V_i=p^{-1}(U_i)$  with coordinates $(t,(x_i:y_i))$ by  $f_i(t,(x_i:y_i))=(t,F_i(t)(x_i:y_i)).$
  Let $z_i=\frac {y_i}{x_i}\in\ov \BC,$ and $$S_1\cap U_i=\{(t,z_i=a_i(t)\}, \ S_2\cap U_i=\{(t,z_i=b_i(t)\}.$$
  Since $F_i(t)=\psi_{U_i,f}(t)$ depend on $t$   holomorphically,    $a_i(t)$ and $b_i(t)$ are meromorphic  functions in $U_i.$ Since $S_1\cap  S_2=\emptyset,$  $a_i(t)\ne b_i(t)$   for all $t\in U_i  $  and all $i.$

The holomorphic  coordinate change  in $V_i$ introduced in 
item (1) of the proof of \propref{fixedpoints} is 
$$(t,z_i)\to (t, \frac {z_i-a_i(t)}{z_i-b_i(t)})=(t,u_i).$$ 

 In these coordinates $S_1\cap V_i=\{u_i=0\},$ and    $S_2\cap V_i=\{u_i=\infty\}. $
Since   both sections are globally defined and  $f$-invariant,  there are holomorphic functions $\mu_{i,j}\in U_i\cap  U_j, \mu_{i,j}\ne 0,$ such that 
$$(t,u_j)=\Phi_{i,j}(t,u_i)=(t,\mu_{i,j}u_i).$$
Since $u_j=\mu_{i,j}u_i=\mu_{k,j}u_k$  in $U_i\cap U_j\cap U_k,$ we have $\mu_{i,k}=\mu_{j,k}\mu_{i,j}$, that is we have a cocycle. It defines 
 a holomorphic line bundle $\mathcal L_f$  on $T$ with transition functions $ \mu_{i,j}$  such that $X\setminus S_2$ is the total  body of   $\mathcal L_f$ and $S_1$ is the zero section of $\mathcal L_f.$     Moreover,  (see item  (1) of proof of \propref{fixedpoints})
$$f(t,u_i)=(t,\lambda_fu_i) , \ \lambda_f\ne 0.$$ 
If $\mathcal L_f$  had another section, then the $\BP^1-$bundle $X$ would have three disjoint sections, thus would be isomorphic to $T\times\BP^1 $  (excluded case).   Since  every $g\in \Aut(X)_p$ of type {\bf A} has sections as the set of fixed points, it has  to have  the same Data   $(S_1,S_2). $

The maps having the same Data differ   only  by the coefficient
 $\lambda_f\in\BC^*.$ 
  It follows that an automorphism of type {\bf A} is uniquely defined by its restriction to any fiber $P_t, t\in T$  (cf. Lemma 4.3 of \cite{kostya}). On the other hand,  for  every $\lambda\in\BC^*$ one can define  an automorphism 
$ f_{\lambda}\in\Aut(X)_p$  of  type {\bf A} on $X$ by formula: \begin{equation}\label{lf}  (t,u_i)=(t,\lambda  u_i). \end{equation} Thus  all the automorphisms of  type {\bf A} on $X$ form a subgroup $\Gamma_A\cong\BC^*$  of $\Aut(X)_p.$ \end{proof}

{\bf Case B.} If the  set of all fixed points of non-identity  $ f\in\Aut(X)_p$  is a section $S$ of  $(X,p,T)$
we will say that $f$ has type  {\bf B} with Data   $S.$ 
\begin{Lemma}\label{caseB}  Assume that $ f \in\Aut(X_p), \ f\ne id,$ has type {\bf B} with   Data $S$ and $X\not\sim T\times\BP^1.$
Then\begin{itemize}\item $X\setminus S$ is an $\BA^1-$bundle $\AA_f$ over $T;$
\item  $\mathcal A_f$ has no sections;
\item $\Aut(X)_p$  contains a subgroup $\Gamma_B\cong \BC$ of all $ g\in \Aut(X)_p$  with the same Data  $S;$
 \item any  automorphism $ g\in \Aut(X)_p$ of type {\bf B} belongs to $\Gamma_B;$
 \item an  automorphism $g\in\Gamma_B$ is uniquely determined by its restriction to any fiber $P_t$ with $t\in T;$  
\item $\Aut(X)_p$  contains no automorphisms of type {\bf A}.\end{itemize}\end{Lemma}
\begin{proof}
   In notation of \propref{fixedpoints}  in this case    $\delta_f=TD_f-4=0.$   Thus \eqref{fixed2} has the set of solutions $S=\{2y_if_{i,12}+x_i( f_{i,11}-f_{i,22})=0\}\subset X$ of fixed points of
 $f.$   Consider the set $R_f\subset T$ defined locally  by conditions \begin{equation}\label{113}f_{i,12}(t)=f_{i,21}(t)=0, \ f_{i,11}(t)=f_{i,22}(t).\end{equation}  Since $f\ne id,$  $R_f$  is  an analytic subset of $T,$ and  $\codim R_f\ge 2.$  Note that $p^{-1}(R_f)\subset S.$  Consider the function 
  $$g_i(t)=\frac{f_{i,22}(t)-f_{i,11}(t)}{f_{i,21}(t)}=\frac{2f_{i,12}(t)}{f_{i,11}(t)-f_{i,22}(t)}$$
  (the equality follows from  $\delta_f=0$). Function $g_i$ is meromorphic in $ U_i\setminus R_f.$  Since $\codim R_f\ge 2,$
  by the Levi's   Theorem (\cite{Levi},  \cite[Theorem 4 Chapter   VII]{Narasimhan}, \cite[Section 4.8]{Fi})  $g_i$  may be extended to a meromorphic function to $U_i.$ 

       Define $w_i=\frac  {y_i}{x_i+y_ig_i(t)}\in\ov \BC.$   The direct computation shows that $f(t,w_i)=w_i+a_i(t),$  where $a_i=\frac{2f_{i,21}(t)}{\tr (\tilde F_i(t))}.$  Since $\delta_f=0$ the denominator never vanishes, thus $a_i(t)$ is a holomorphic function in $U_i.$
       
       The set  $\{a_i(t)=0\}=R_f\cap U_i$ has codimension 1  in $U_i,$  which is impossible  if   $R_f\ne\emptyset. $
 It follows that $R_f=\emptyset.$  Thus, 
 $f \bigm  |_{P_{t}}\ne id$ for any $t\in T$ and  $a_i(t)$ does not vanish    in $ U_i.$

Since $S=\{w_i=\infty\}$ is globally defined, we have   $w_j=A_{i,j}(w_i)=\nu_{i,j}w_i+\tau_{i,j},$ where $\nu_{i,j}$ and   $\tau_{i,j}$ are holomorphic  functions  in $U_i\cap  U_j.$  
Since $f$ is globally defined 
$$\nu_{i,j}(w_i+a_i(t))+\tau_{i,j}= (\nu_{i,j}w_i+\tau_{i,j})+a_j(t),$$
we have 
\begin{itemize}
\item $\{\nu_{i,j}\}$  do not vanish in $U_i\cap  U_j $ and form a cocycle, thus define a line holomorphic bundle $\mathcal M_f$  on $T;$
\item   $\{a_i(t)\}$ is a section of  $\mathcal M_f.$
\end{itemize}
 Since a nontrivial  holomorphic line bundle on $T$ has no nonzero sections, either $a_i(t)\equiv 0$ and $f=id$  (the excluded case),  or  $\mathcal M_f$  is trivial and   we have a global holomorphic, hence constant function $$a_i(t)\equiv a_f.$$ 
    Thus,\begin{itemize}\item  $X\setminus S$ is an $\BA^1$-bundle   $\mathcal A_f$  with transition  holomorphic functions  $\tau_{ij} $ in $U_i\cap  U_j;$
\item  for every $b\in \BC$  there is $f_b\in\Aut(X)_p$ defined in each $V_i$ by $$f_b(t,w_i)=(t,w_i+b);$$
\item  The subgroup $\Gamma_B$  of all $f_b, \ b\in \BC$ is isomorphic to $\BC^+.$\end{itemize}

Let us show that $\mathcal A_f$   has no sections. If it had a section $S_1,$ then $S, S_1, f(S_1)$ would be three disjoint sections of $X.$  Since  $X\not\sim  T\times\BP^1,$ this is impossible. It follows that $\Aut(X)_p$ does not contain neither an automorphism of type {\bf A} or an automorphism of type {\bf B} with Data different from $S.$

   The maps having the same Data  $S$ differ   only  by the summand
 $a_f\in\BC.$  It follows that an automorphism of type {\bf B} is uniquely defined by its restriction to a fiber $P_t$ for every $t\in T..$ \end{proof}

{\bf Case C. } Assume that $X\not\sim  T\times\BP^1$ and   the  set $S\subset X$  of all   fixed points     of non-identity map $ f\in\Aut(X)_p$ is a smooth unramified double cover of $T.$ We will call such an $f$ an automorphism  of type {\bf C }  defined by Data $S.$   

Consider $$\tilde X:=\tilde X_f:=S\times_T  X=\{(s,x)\in S\times  X \subset X\times  X  \ \ :  p(s)=p(x)\}.$$ We denote the  restriction of $p$  to  $S$
by the same letter  $p,$  while   $p_X$ and $\tilde p$ stand for 
 the restrictions to $\tilde X$ of natural projections $ S\times  X\to X$ and $ S\times  X\to S $ respectively.  We write
$\invl:S\to S$ for  an involution (the  only non-trivial deck transformation for $ p \bigm |_S$).

We have

 \begin{itemize}

\item [a)]  
The following 
 diagram commutes
\begin{equation}\label{diagram41}
\begin{CD}
 \tilde X  @>{p_X}>>X\\
@V\tilde p VV@V pVV\\
S\subset X  @>{p}>>T
\end{CD}.\end{equation}
\item [b)] $p_X:\tilde X\to X $ is an unramified double cover of $X;$
\item [c)] Every fiber  $\tilde p^{-1}(s), s\in S$ is isomorphic to $$P_{p(s)}=p^{-1}(p(s))\sim \BP^1;$$
\item [d)]  $\BP^1$-bundle $\tilde X$ over $S$ has two sections: $$S_+:=S_+(f):=\{(s, s)\in \tilde X, \  s\in S\subset X\}$$ and $$ S_-:=S_-(f):=\{(s, \invl(s))\in \tilde X,  \ s\in S\subset X\}.$$ They are mapped onto $S$ isomorphically by $p_X.$
\item[e)] Every section $N=\{t,\sigma(t)\}  $   of $p$ in $X$  induces the section
 $\tilde N :=\{(s,  \sigma(p(s))\}$  of  $\tilde p $  in $\tilde X.$   
We have   $p_X(\tilde N)=N$ is a section of $p,$ 
  thus  $\tilde N$ cannot coincide  $ S_+$ or $ S_-.$
\item [f)]  Every $h\in \Aut(X)_p$ induces an automorphism $\tilde h\in \Aut(\tilde X)_{\tilde  p}$ defined by $$\tilde h(s,x)=(s,h(x)).$$ 
 \item [g)] In particular, for  the lift $\tilde f$ of $f$  all the
 points of  $S_+$ and $S_-$ are fixed, hence, $\tilde f$ is of type {\bf A}  with Data $( S_+, S_-).$ 
 \item[h)]   The map $\tilde f$ is uniquely determined by its restriction to any fiber $\tilde P_s=\tilde p^{-1}(s)$ (see case {\bf A}), hence $f$ is uniquely determined by its restriction
on the   fiber $ P_t= p^{-1}(t).$  Indeed, if $f\bigm |_{P_t}=id,$  then

- \  $\tilde f\bigm |_{P_s}=id,$ hence 

- \  $\tilde f=id,$  hence 

- \ $\tilde f\bigm |_{P_{s_1}}=id$  for every $s_1\in S,$  hence 

 - \  $f\bigm |_{P_{t_1}}=id$   for $t_1=p(s_1)\in T.$
 
\item[i)] It follows that $h\mapsto\tilde h$ is a   group  embedding of $\Aut(X)_p$ to $\Aut(\tilde X)_{\tilde p}.$ 

 \item [j)]   Involution  $s\to \invl(s)$ may be extended from $S$ to a  holomorphic involution $\tilde X$ by  $$\invl(s,x)= (\invl(s),x);$$
\item[k)]   $S$ is a poor manifold  by \lemref{cover}. 
\end{itemize}
Clearly, the maps having the same Data differ   only  by the coefficient
 $\lambda_{\tilde f}\in\BC^*.$ 
 
 \begin{corollary}\label{csection} If $X$ admits a  non-identity   automorphism of type {\bf C}
and   $\tilde X_f\not \sim S\times \BP^1$   then  $\BP^1-$bundle $p:X\to T$   does not have  a section. In particular, it does not admit automorphisms of type  {\bf  A} or {\bf B}.\end{corollary}
\begin{proof} Indeed,  if  $X$ admitted  an automorphism  of type  {\bf  A} or {\bf B}  then, by \propref{fixedpoints}, there would be  a section  $S$ of $p$  in $X.$  The preimage $p_X^{-1}(S)\subset \tilde X$
would be  a section $\tilde S$   of $\tilde p$  in $\tilde X_f.$  
 Thus 
 $\tilde X_f$ would  admit three disjoint sections:    $S_-, S_+,$ and $\tilde S. $  In this case  $\tilde X_f$ would be the direct product $S\times\BP^1.$ \end{proof}

\begin{Lemma}\label{caseAC} Assume that $ f \in\Aut(X)_p, \ f\ne id,$ has type {\bf A} with   Data $ (S_1,S_2)$ and $X\not\sim T\times\BP^1.$  Let $\LL_f$ be defined by $f$  (see Case {\bf A} for   the definition) holomorphic  line bundle on $T$   with transition functions $\mu_{ij}$ and such that $S_1 $ is its zero section. 
%%%%%%%%%%%%%%%
%%%%%%%%%%%%%%%%%%
Then one of the following holds:\begin{enumerate}
\item   $\Aut(X_p)=\Gamma_A\cong \BC^*$  and  $X$ admits  only  automorphisms of type {\bf A}  except   the $id.$; 
\item $\Aut(X_p)$ contains  an automorphism $h$  of type {\bf C}  with Data $S.$ In this  case $\LL_f^{\otimes 2}$  is a trivial holomorphic line bundle and the corresponding to $h$ double cover   $\tilde X_h\sim S\times\BP^1.$\end{enumerate}\end{Lemma}
\begin{proof}

(1)  By \lemref{caseB} we know that $\Aut(X)_p$ contains no automorphism of type {\bf B}. If there is  no automorphism of type {\bf C}  then  all $f\in\Aut(X)_p$  are of type {\bf A}  except the $id.$. By \lemref{caseA} in this case $\Aut(X_p)=\Gamma_A\cong \BC^*.$

(2) Let $h\in \Aut(X)_p$  be   of type {\bf C}  with Data $S.$   Let a point $t\in  U_i\subset T,$  where $U_i$ is a fine covering of $T,$  and let $(u, t_i), \  u\in U_i, \  t_i\in\ov\BC$ be coordinates   in $V_i=p^{-1}(U_i)\subset X.$    Since $S_1\cup S_2$ are the only sections of $p:X\to T$ and points of $S_1\cup S_2$ are not fixed  by $h,$ we have
\begin{equation}\label{onfibre}h(u,t_i)=\frac{\nu_i(u)}{t_i}\end{equation}
where $\nu_i$ are holomorphic in $U_i$ functions.  Since $h$ is defined globally, we have
 $$\mu_{ij}(u)(\frac{\nu_i(u)}{t_i})=\frac{\nu_j(u)}{\mu_{ij}(u)t_i}.$$

Thus $\mu_{ij}(u)^2\nu_i(u)=\nu_j(u).$ Since  non-trivial line bundle over $T$ has only zero section, it follows that $\LL_f^{\otimes 2}$ is a trivial bundle and $\nu_i(u)=\nu$ is a constant function.  The section $S$ is defined locally by equation $t_i^2=\nu.$   By \corref{csection},
$\tilde X\sim S\times\BP^1.$ \end{proof}

 Assume that $ f\in \Aut(X)_p,\  f\ne id,$ and $f$ is of type {\bf C} defined by Data (bisection) $S.$  
Let $\tilde X:=\tilde X_f$  be  the corresponding double cover (see case {\bf C} in   \secref{p1}  and  diagram  \eqref{diagram41}).
Recall that $S$ is poor and $\tilde p:\tilde X\to S$ has two sections 
% Let $z:S\times \BP^1\to \BP^1\sim \ov\BC_z$ be the natural projection.
%Then the coordinate $z$ in  $\BP^1\sim \ov\BC$  is  globally defined in $\tilde X.$
 $S_+=\{(s,s), \ s\in S\}$ and $S_-=\{(s,\invl(s)) \ s\in S\}.$

\begin{Lemma}\label {Yf}  Assume that $ f\in \Aut(X)_p,\  f\ne id,$ and $f$ is of type {\bf C} defined by Data (bisection) $S.$

(1) If the corresponding double cover (see case {\bf C}) $\tilde X:=\tilde X_f$  is  not isomorphic to $S\times \BP^1$ then $\Aut(X)_p$ 
has exponent 2 and consists of $2$ or 
 4 elements.

 (2)
   If  $\tilde X:=\tilde X_f$ is isomorphic to $S\times \BP^1$  then there  are two  sections  $S_1,S_2\subset X$ of $p.$  Moreover, $\Aut(X)_p$  is  a disjoint union of its abelian complex Lie subgroup $\Gamma\cong \BC^*$ of index 2 and its coset 
   %%%%%%change May 16:  ` to {\prime}%%%%%%%%%%
   $\Gamma^{\prime}.$ Subgroup $\Gamma$ consists in those $f\in\Aut(X)_p $ that fix $S_1$ and $S_2.$  Coset  $\Gamma^{\prime}$ consists in those $f\in\Aut(X )_p$ that interchange $S_1$ and $S_2.$
\end{Lemma}

\begin{proof}  Choose a point  $a\in S. $ Let $b=p(a)\in T.$ It means that $a$ is one of two points in $S\cap P_b.$
  The lift $\tilde f$ of $f$ onto $\tilde X $  has type {\bf A}, and for the corresponding line bundle $\tilde \LL_{\tilde f}$  
  we may assume that  $S_+$ is a zero section.  Let \begin{itemize}
 \item
  $\tilde U_i$ be the fine covering of $S;$
  \item  $\mu_{ij}$ be  transition functions of $\tilde \LL_{\tilde f}$  in $\tilde U_i\cap  \tilde U_j;$ 
   \item $\tilde V_i=\tilde p^{-1}(\tilde U_i)\subset \tilde X; $
  \item  $(u, z_i)$ be the local coordinates in $\tilde V_i$ such that $z_j=\mu_{ij}z_i$ in $\tilde U_i\cap \tilde  U_j.$
  \item  $a\in \tilde U_i, \ \invl(a)\in \tilde U_k$ and $\tilde U_k\cap \tilde U_i=\emptyset;$ 
  \item $b=p(a)=p(\invl(a))\in T.$
  \end{itemize}

 Since  $S_+$ is the  zero section, $z_i=0$  on $S_+\cap\tilde V_i, $  $z_i=\infty $  on $S_-\cap\tilde V_i, $  while $z_k=0$  on $S_+\cap\tilde V_k, $ and $z_k=\infty $  on $S_-\cap\tilde V_k.$
 We have 
 \begin{equation}\label{111}
 z_i(a,\invl(a))=\infty,  \ z_i(a,a)=0 ,  \end{equation}
  \begin{equation}\label{101}
  \ z_k(\invl(a),a)=\infty,  \ z_k(\invl(a),\invl(a))=0.  
 \end{equation}

\begin{comment} On the other hand,  $a$ and $\invl(a)$ are the intersection points of $S$  with $p^{-1}(b)= P_b\subset X$
 and for $p_X:\tilde X\to X,$ by construction, we have:
 \begin{equation}\label{112}
 (b, \invl(a))=p_X(a,\invl(a))=p_X(\invl(a),\invl(a)), \end{equation}   \begin{equation}\label{102} (b,a) =p_X(a,a)=p_X(\invl(a),a).  
 \end{equation}
\end{comment} 
 
 It may be demonstrated by the following diagram:
\begin{equation}\label{diagram4}
%\begin{CD}
\begin{CD}
 P_b  @>{(a,id)}>>a\times P_b@>{z_i}>>\ov \BC_{z_i}\\
@V id VV@V{ }VV@V \alpha VV\\
P_b  @>{ (\invl(a),id)}>>\invl(a)\times P_b@>{z_k}>>\ov \BC_{z_k}
\end{CD}.\end{equation}
 Here the isomorphism    $\alpha:\ov\BC_{z_i}\to \ov\BC_{z_k}$    is defined in such a way that the diagram is commutative.

We get from \eqref{111}  and \eqref{101}  that 
$\alpha(0)=\infty, \ \alpha(\infty)=0.$ Hence  
$$
z_k=\alpha(z_i)=
\frac{\nu }{z_i}$$
for some $\nu\ne 0.$  By construction
\begin{equation}\label{114}p_X(\invl(a), \alpha(z_i))=p_X(a,z_i).\end{equation}

Consider an automorphism $h\in\Aut(X)_p.$  Let $\tilde h$ be its pullback to 
$\Aut(\tilde X)_{\tilde p}$  defined by $\tilde  h(s,x)=(s,h(x)).$ 
Let $n_1(z_i)=\tilde h \bigm |_{\tilde  P_a}, $  which  means that $h(a,z_i)=(a,n_1(z_i)).$ 
Let $n_2(z_k)=\tilde h \bigm |_{\tilde P_ {\invl(a)}}, $  which  means that $h( \invl(a),z_k)=(a,n_2(z_k)).$ 
 Choose in $P_b$ the  coordinate $z$  such that     $z_i=p-X^*(z),$  i.e., $  p_X(a,z_i)=(b, z_i)$  for a point $(a,z_i)\in\tilde P_a .$
 By construction,  $ z(a)=0,z(\invl(a))=\infty.$

 We have the following commutative diagram:
%\begin{equation}\label{diagram44}
%\begin{CD}
$$\begin{CD}
   P_b\ni     (b,z_i)  @>{(a,id)}>>(a,z_i)@>{\tilde h}>>(a,n_1(z_i))@>{p_X}>>(b,n_1(z_i))\in  P_b\\
@V id VV@V{ (\invl(a),\alpha) }VV@V {(\invl(a),\alpha)} VV@ |  \\
P_b \ni   (b,z_i) @>{ (\invl(a),id)}>>(\invl(a),\al(z_i))@>{\tilde h}>>(\invl(a),n_2(\al(z_i)))@>{p_X}>>(b,\al(n_1(z_i))\in  P_b
\end{CD}.$$
%\end{equation}

 Hence 
 
 \begin{equation}\label{116} \frac {\nu}{n_1(z_i)}=         \al(n_1(z_i))=n_2(\al(z_i))=n_2(\frac{\nu}{z_i}).\end{equation}
 
 (1) Assume that $\tilde X\not\sim S\times\BP^1$.
 
  It follows from  \eqref{onfibre}   and  the  proof of \lemref{caseAC}  
  (applied to  $\tilde X$)   that 
 for every $\tilde h \in \Aut(\tilde X)_{\tilde p}$  in every $U_j$ of our fine covering either 
  $\tilde h(s,z_j)=\lambda  z_j, $ 
  or $h(s, z_j)=\frac{\lambda }{ z_j}$ for some $\lambda\in\BC^*,$  and   $\lambda$ does not depend on $s$  or $j.$

Thus, one of following two conditions holds:

  (a)   $n_1(z_i)=\lambda z_i, \ n_2(z_k)=\lambda z_k,$ \  $z_k=\frac{\nu}{z_i}$  and from  \eqref{116} 
$$\frac {\nu}{\lambda z_i}=\lambda\frac{\nu}{z_i}$$  

(b)  $n_1(z_i)=\frac{\lambda}{ z_i}, \ n_2(z_k)=\frac{\lambda}{ z_k},$ $z_k=\frac{\nu}{z_i}$ and   from  \eqref{116} 
$$\frac {\nu z_i}{\lambda}=\frac {\lambda z_i}{\nu}.$$
In the former case $\lambda=\pm 1, $    in the latter  $\lambda=\pm \nu.$ Hence, at most 4 maps are possible. Clearly,  the 
 squares of all these maps are the identity map.

(2)  Assume that $\tilde X\sim S\times\BP^1$.
 Let $z:S\times \BP^1\to \BP^1\sim \ov\BC_z$ be the natural projection.  
%Then the coordinate $z$ in  $\BP^1\sim \ov\BC$  is  globally defined in $\tilde X.$
Since $S_+=\{(s,s), \ s\in S\}$ and $S_-=\{(s,\invl(s)) \ s\in S\}$  have algebraic dimension 0, the rational function $z$
is constant along these  sections. 
We may assume  that  $z=0$ on $ S_+=\{(s,s)\}$ and $z=\infty$ on $S_-=\{(s,\invl(s)\}$   and all $z_j=z.$ Moreover, in this case $n_1(z):=n(z)=n_2(z)$ and 

\begin{equation}\label{th}  \tilde h(s,z)=(s,z'), \text{where} \  z':=n(z):=\frac{az+b}{cz+d}, \ a,b,c,d\in\BC.\end{equation}

On the other hand, it follows   from  \eqref{116}  that the map $\tilde h(s,z)$  defined by  \eqref{th} may be pushed down to $X$ if and only if 
$$\alpha(n(z))=n(\alpha(z)).$$
In expression $\alpha (z)=\frac{\nu}{z},$ we may assume  that $\nu=1.$  (Indeed, choose  a $\sqrt{ \nu}$  and divide  $z$ by it).
The map $\tilde h(s,z)$  defined by  \eqref{th} may be pushed down to $X$ if and only if 
\begin{equation}\label{down} \frac{a\frac{1}{z}+b}{c\frac{1}{z}+d}=\frac{cz+d}{az+b}.\end{equation}

 For every $(a:b)\in \BP^1, a^2-b^2\ne 0 ,   $     two  types  of $\tilde h$ with  property \eqref{down}  are  possible:
\begin{equation}\label{gamma} 
 z'=\frac{az+b}{bz+a}=\tilde  h_{a,b}(z), \end{equation}  and 

\begin{equation}\label{gamma1} 
 z'=-\frac{az+b}{bz+a}=-\tilde  h_{a,b}(z)=\tilde  h_{a,-b}(-z). \end{equation}

Note that the only nontrivial automorphism of $\tilde X$ leaving $z=0,z=\infty$ invariant is $-\tilde  h_{a,0},$
which is  the lift $\tilde f$ of $f.$
  All the transformations $ h_{a,b}$ form an abelian group $\tilde \Gamma$ with 
$$ h_{a,b} h_{\al,\be}=h_{c,d}, \ c=a\al+b\be,  \ \ d=a\be+b\al.$$
The transformations $ -h_{a,b}$ form  a coset $\tilde \Gamma'=-h(1:0)\tilde\Gamma.$ 
All the transformations from $\tilde \Gamma\cup\tilde \Gamma'$ may be pushed down to $X.$ We have:  $\Aut(X)_p$
is embedded into  $\Aut(\tilde X)_{\tilde p} $ and its image is $\tilde \Gamma\cup\tilde \Gamma'.$ 
Thus 
$\Aut(X)_p$ is the disjoint union of a subgroup $\Gamma$
and its coset $\Gamma'$   corresponding to $\tilde \Gamma$ and $\tilde \Gamma'$  respectively. 
The index of 
$\Gamma$   in   $\Aut(X)_p$   is 2.

 Note that the sets $\{z=1\}$ and $\{z=-1\}$ consist of  fixed points of all the maps $\tilde h_{a,b}$ if $b\ne 0.$ Moreover, they are invariant under deck transformation $(s,z)\to (\invl(s), \frac{1}{z}).$  
 Their images provide  two sections $S_1, S_2$
 of $\BP^1-$bundle $p:X\to T.$ 
Hence, in this case $X=\BP(E)$ for some decomposable rank two vector bundle $E$ over $T.$
If we change coordiantes $w=\frac{z+1}{z-1}$ then $w'=\frac{z'+1}{z'-1}=\tilde h_{a,b}(w)=w\frac{a+b}{a-b}:=w\mu_{a,b}$ and $ \tilde h_{a,b}(\tilde h_{\al,\be}w)$  corresponds to $\mu_{a,b}\mu_{\al,\be}.$  The condition $a^2-b^2\ne 0$ means that   $\mu\ne  0,\infty.$ Thus, $\Gamma\cong\BC^*$
 as a complex Lie group.
In coordinates $w$ we have $-\tilde h_{a,b}(w)=\frac{1}{w\mu_{a,b}}$  thus $\Gamma'$ consists of maps, interchanging the sections. 
\end{proof}.

 \begin{prop}\label{ap} Let $(X,p,T)$ be a $\BP^1-$bundle over a poor manifold $T.$ Then one of the following holds:
 \begin{enumerate}\item $X\sim T\times\BP^1;$
 \item $\Aut(X)_p$ has exponent at most 2 and consists of 1,2 or 4  elements; 
 \item $\Aut(X)_p\cong\BC^+;$ 
 \item $ \Aut(X)_p\cong\BC^*;$ 
 %change May 16: ` to {\prime} \cup yo \sqcup
 \item $\Aut(X)_p=\Gamma\sqcup\Gamma^{\prime}$ where $\Gamma\cong\BC^*$   is a complex Lie subgroup of 
 $\Aut(X)_p$ and $\Gamma^{\prime}$ is its coset in $\Aut(X)_p.$ \end{enumerate}\end{prop}
 %%%%%%%%%%%%%%%%%%%
 \begin{proof}  We use the following: assume that $X\not \sim T\times\BP^1$  and $f\in\Aut(X)_p, \ f\ne id.$  Then \begin{itemize}
 \item $f$ being of type {\bf A} implies the existence of exactly two sections of $p$    (see Case {\bf A});
 \item   $f$ being of type {\bf B} implies the existence of exactly one section of $p$  (see Case {\bf B}); 
 \item  $f$ being of type {\bf C} implies the existence of either no    or  exactly two sections of $p$   (see Case {\bf C}  
 and   \lemref{Yf}).\end{itemize}
 Consider  the  cases.

 (2) If $X$ contains no sections of $p$ then either $\Aut(X)_p=\{id\}$  or there is $f\in \Aut(X)_p$ of type {\bf C}.  Let $S$ be a bisection of $p$
 that is the fixed point set of $f.$  The corresponding to $f$ double cover $\tilde X_f$ of $X$  cannot be isomorphic to $S\times\BP^1$  by \lemref{Yf} (2),  since there are no sections of $p.$  Thus, by \lemref{Yf} (1), 
 $\Aut(X)_p$ has exponent at most 2 and consists of 2 or 4  elements.
 
 (3) Assume that $X$ contains exactly one section $S$ of $p.$ Then $\Aut(X)_p=\{id\}$  or there are 
 %%%change May 16: insert non-identity
 non-identity
 %%%%%%%%%%%
 $f\in \Aut(X)_p$  of type  {\bf B}  only.   Then $\Aut(X)_p\cong\BC^+$ by  combination of  \lemref{caseB}, \lemref{csection},   and  \lemref{Yf}.   

(4)  Assume that $X$ contains exactly two sections  $S_1$ and $S_2$ of $p.$  Then there are two options:\begin{enumerate}\item  $\Aut(X)_p$  consists of 
automorphisms of type  {\bf A } only (except $id$)  and   $\Aut(X)_p\cong \BC^*$  according to \lemref{caseA};
\item $\Aut(X)_p$  contains 
automorphisms of type  {\bf A }  and  {\bf C}.   By \lemref{Yf},   $\Aut(X)_p=\Gamma\sqcup\Gamma'$ where $\Gamma\cong\BC^*$ is  a
 complex Lie subgroup of 
 $\Aut(X)_p$  consisting in  those maps that fix $S_1$ and $S_2, $  and $\Gamma'$ is its coset in $\Aut(X)_p$ 
 %May 16: and replace by that
 that consists in maps that interchange the sections.\end{enumerate}\end{proof}
%%%%%%%%%%%%%%%%

\begin{remark}\label{open}  Let us formulate a  byproduct of the proof of 
%May 16: the deleted
 \propref{fixedpoints}.  Assume that $(V,p,U) $ is a $\BP^1-$bundle over   a connected complex (not necessarily   compact)   manifold $U,$ and  let $f\in\Aut(V)_p% May 16: insert ,f ne id
 , f \ne id.$   Then 
\begin{itemize}\item  
the function $\TD(u)$ is globally defined;
\item If $\TD(u)= \mathrm{const}\ne 4$  on $U$ then  the set of fixed points of $f$ is an unramified (may be reducible) double cover of $U;
$\item   If $\TD(u)\equiv  4$  on $U$ and $U$ contains no  analytic  subset of  codimension 1, then  the set of fixed points of $f$ is a  section of $p.$
\end{itemize}
\end{remark}
\section { $\BP^1-$bundles over poor K\"{a}hler manifolds} \label{abelian}

In this section we continue to consider a triple $(X,p ,T)$  that is a $\BP^1-$bundle   over a 
poor manifold   $T.$  Further on we assume that $T$ is a   {\bf  K\"{a}hler manifold.}
Recall that this means that \begin{itemize}\item $X$ and $T$ are connected complex compact manifolds;
\item $T$ contains no rational curves and no  analytic subspaces  of codimension 1  (in particular, $a(T)=0$);
\item $T$ is K\"{a}hler; 
\item $p:X\to T$ is a surjective holomorphic map;\item $X$ is   holomorphically locally trivial fiber bundle over $T$ with fiber $\BP^1$ and projection $p.$
\end{itemize}
%%%%%%%%%%%%%%%%%%%%%%%%%%%%%%%
%%%%%%%%%%%%%%%%%%%%%%%%%%
\begin{Lemma}\label{  Kaehler  } If $T$ is a  poor   K\"{a}hler manifold and $\Aut(X)_p\ne\{id\}$ then $X$ is a   K\"{a}hler   manifold.\end{Lemma}
\begin{proof}  Let $f\in \Aut(X)_p, f\ne id.$  Then either $X$ or its \`{e}tale double cover $\tilde X$ is 
$\BP (E)$ where $E$ is a holomorphic rank two vector bundle over a   K\"{a}hler   manifold $T$ or its double cover, respectively, (that is also   K\"{a}hler, see \lemref{cover}).
In both  cases $X$ is  K\"{a}hler   according to (\cite {Voisin}, Proposition  3.18). \end{proof}

\begin{Corollary}\label{Jordan}  $\Bim(X)=\Aut(X)$ is Jordan. \end{Corollary}
\begin{proof} The statement  follows from the result of  Jin Hong   Kim, \cite{Kim}.
\end{proof}

%%%%%%%%%%%%

\begin{Lemma}\label{commutative}   Consider  a short exact sequence  of connected complex Lie groups:
$$0\to A\overset {i}{\rightarrow} B\overset {j}{\rightarrow} D\to 0.$$ Here $i$ is a closed holomorphic embedding and $j$ is surjective holomorphic. Assume that $D$ is a complex  torus and $A$ is isomorphic as a group Lie either to $  \BC^+$ or to $\BC^*.$  Then $B$ is commutative.\end{Lemma}
\begin{proof}  
   {\bf Step 1.}  First, let us prove    that $A$ is a central subgroup  in B. Take any element $b\in  B.$ 
 Define  a holomorphic map $\phi_ b:A\to A, \ \phi_b(a)= bab^{-1}\in A$
for an element $a\in  A.$  Since  it depends holomorphically on $b,$  
we have  a holomorphic map $\xi :B\to \Aut(A),  b  \to \phi_b.$

 Since $A$ is commutative, for every $a\in A$
we have $\phi_{a b}=\phi_b.$ Thus   there is a well defined map $\psi$ fitting into the following commutative diagram \begin{equation}\label{kawamata}
\begin{aligned}
&        &   &                 &  {B}                              &                  &      &             & \notag \\
&        &j  & \swarrow  &                                        &{\searrow} &\xi   &               &  \notag \\
& D&   &                &\stackrel{\psi}{\rightarrow} &                  &       &\Aut(A)&
      \end{aligned},  \end{equation}
The map $\psi=\xi\circ j^{-1}$ is  defined at every point of $D.$  
 It is holomorphic (see, for example, \cite{OV}, \$ 3).
Since $D$ is a complex torus,   we have $\psi(D)$ is $\{id\}.$  It follows that $A$ is a central subgroup of $B.$

     {\bf Step 2.}   Let us prove  that $B$ is commutative.  Consider a  holomorphic  map  $\mathrm{com} :B\times B\to A$  defined by 
$\mathrm{com}(x,y)=xyx^{-1}y^{-1}.$  Since $A$ is a central subgroup of $B, $  similarly to    {\bf Step 1}
we get a holomorphic map  $D\times D\to A.$  It has to be constant since $D$ is a complex  torus and  $A$ is either $\BC^+$ or $\BC^*.$\end{proof}

\begin{Theorem}\label{main1} 
Let $X$  be a  $\BP^1$-bundle over a  K\"{a}hler poor manifold  $T$ and $X\not\sim  T\times \BP^1.$
 Then    the connected identity component $ \Aut_0(X)$ of $\Aut(X)$  is commutative and  the quotient $\Aut(X)/\Aut_0(X)$  is a bounded group.
\end{Theorem}
\begin{proof}  

 %%%%%%%%%%%%%%%%%%%
%%%%%%%%%%%%%%%%%%%%%%%%
%%%%%%%%%%%%%%%%% 
  From   equation \eqref{Chevalley}, applied to $X$ and $T,$  combined with \lemref{tauh}  and \remarkref{Fu}, we get the following commutative diagram 
  of  complex Lie groups and their   holomorphic homomorphisms.
 
\begin{equation}\label{Chevalley2} 
\begin{aligned}
&    0    &  \to &    L(X)           &  \to    &  \Aut_0(X)   & \to     &   \Tor(X)          &\to &0& \notag \\
&          &       & \downarrow &           & \downarrow\tau &    &   \downarrow&    &  &  \notag \\
& 0       &  \to & \   0             &\to         &  \Aut_0(T)  &    \cong  &    \Tor(T)         &\to &0&    \end{aligned},  \end{equation}
 
 Let us  identify a complex  torus with the group of its translations  and   put
$H:=\tau(\Aut_0(X)).$  Then $H$ is the image of a
%May 16: delete a
  complex  torus  $\Tor(X)\cong \Aut_0(X)/  L(X) $ under a holomorphic homomorphism,  thus is a complex subtorus of $\Tor(T).$  
    Let $G$ be the preimage of $H$ in $\Aut(X), $ with respect to 
  %May 16: change!
  $\tau: \Aut(X) \to \Aut(T)$. By definition, $G$ is a a complex Lie group that contains $\ker(\tau)=\Aut(X)_p$ as a closed complex Lie subgroup.
  %%%%%%%%%%%%%%%%%%%%%%%%%%
Since $\Aut_0(X)\subset G\subset\Aut(X),$ the identity connected component of $G$ coincides with  $\Aut_0(X).$

One has the following short exact sequences  of complex Lie groups.
%May 16: change  (\Aut(X)_p\cap G)
\begin{equation}\label{short}    1\to \Aut(X)_p \to G  \overset{\tau}{\to}   H\to 1.\end{equation}
\begin{equation}\label{short1}    1\to (\Aut(X)_p\cap\Aut_0(X))\to \Aut_0(X) \overset{\tau}{\to}   H\to 1.\end{equation}
 According to \propref{ap} only the following cases may occur.

{\bf Case  1.}  $\Aut(X)_p$ is finite. Then   $\Aut(X)_p\cap\Aut_0(X)$  is finite as well, hence 
$\Aut_0(X)  \to   H$
 is a  surjective holomorphic homomorphism   of  connected complex Lie groups with finite kernel, thus an unramified finite covering (\cite{OV}, \$4.3). It follows that $\Aut_0(X)$ is a complex torus, hence commutative.

{\bf Case  2.} $\Aut(X)_p\cong \BC^+$ or     $\Aut(X)_p\cong \BC^*.$ 
%%%%%%%May 16: change
 In this case in short exact sequence \eqref{short}   both $H$ and $\Aut(X)_p$  are connected. This implies that $G$ is connected, hence $G=
 \Aut_0(X). $ According to     \lemref{commutative} 
 $\Aut_0(X)=G$ is commutative.
 
 \begin{comment}
 In this case $\Aut(X)_p$ is connected and therefore lies in $\Aut_0(X)$. Hence, \eqref{short1} becomes a short exact sequence
 \begin{equation}\label{short2}    1\to \Aut(X)_p\to \Aut_0(X) \overset{\tau}{\to}   H\to 1.\end{equation}
 On the other hand, the connectedness of both $H$ and $\Aut(X)_p$  combined with  \eqref{short} imply that $G$ is connected as well and therefore lies in $ \Aut_0(X)$. Combining now \eqref{short} and \eqref{short2}, we obtain that
 %Since $H$ is also connected, it follows from \eqref{short}
% that $G$ is connected as well. 
% $G$ is connected.
 %May 16: deletedsince every fiber of surjective  homomorphism  $G\to  H$ is connected.  
 %Thus 
 $G=\Aut_0(X).$  Moreover  the exact sequence \eqref{short} and       \lemref{commutative}  provide that group  $G$ is commutative.  
\end{comment}

{\bf Case  3.}  $\Aut(X)_p$ has a closed subgroup $\Gamma\cong \BC^*$   of index 2.  According to \lemref{Yf} and \propref{ap}
it happens when 
  $X$ admits precisely two 
 sections $S_1, S_2$   of $p$,
and these sections are disjoint. In addition, all
%For 
automorphisms  $ f \in \Gamma$  leave  invariant these  sections   as subsets of  $X.$
 % are fixed.
  As  for automorphisms from coset $\Gamma^{\prime}=\Aut(X)_p\setminus \Gamma$ of $\Gamma$, they interchange $S_1$ and $S_2$.
  %the sections. 

Let show that in this case \begin{equation}\label{gamma100}
\Aut(X)_p\cap\Aut_0(X)=\Gamma. \end{equation}  We will use the following. 

{\bf (a).}  Every automorphism $f\in\Aut(X)$ 
moves a section $S$ of $p$ to a section of $p.$  Indeed,  since $f$ is $p-$fiberwise,  for every $t\in T$ 
we have $$f(S\cap  P_t)=f(S)\cap P_{\tau(t)}.$$
Thus,   since $S$ meets every fiber at  one point, the same is valid for $f(S).$  Since there are only two sections of $p,$ 
\begin{equation}\label{ss} f(S_1\cup  S_2)=S_1\cup S_2.\end{equation}

{\bf (b).}   The action      $\Aut(X)\times X\to X,  \ (f,x)\to f(x),$ on $X$ is holomorphic,    hence continuous.   Thus the
 image   $S$ of a connected set
 $\Aut_0(X)\times S_1$ in $X$ is connected. Since sections $S_1, S_2$ are disjoint, from \eqref{ss}  follows    that 
  $S=S_1$ or  $S=S_2.$  On the other hand, 
 $id\in  \Aut_0(X).$  It follows that $f(S_1)=S_1, \ f(S_2)=S_2$   for every $f\in \Aut_0(X),$ and  
 $\Gamma^{\prime}\cap\Aut_0(X)=\emptyset. $  That proves \eqref{gamma100}
 
 Now\eqref{short1} maybe rewritten as   a short exact sequence
of holomorphic maps of complex Lie groups.

\begin{equation}\label{gamma300}
1\to \Gamma\to  \Aut_0(X)\to H\to 1, \text{  where }\Gamma\cong\BC^*.\end{equation}

  \lemref{commutative} implies that $ \Aut_0(X)$ is commutative. 

Cases {\bf1-3}  give us that $\Aut_0(X)$ is commutative.  
The group $F:=\Aut(X)/\Aut_0(X)$ is bounded  according  to \propref{abounded}.

\end{proof}

Now \thmref{IntroMain1}  follows  from combination of 
\propref{p1bundle},
 \corref{bimaut},
\propref{fixedpoints},
\thmref{main1},  Equations \eqref{short}, \eqref{gamma300}  and 
 \propref{ap}.

\section{Examples of  $\BP^1$-bundles without  sections.}\label{example}

If $S$ is a complex manifold then we write $\mathbf{ 1}_S$ for the trivial line bundle $S\times \BC$ over $S$.

In this section we construct a $\BP^1$-bundle  $(X,p,T)$
 such that\begin{itemize}\item    $T$ is a   complex torus with $\dim(T)=n\ge 2,  \ a(T)=0;$
\item projection $p:X\to T$ has no section, i.e., there is no  divisor
 $ \Delta \subset X$ that meets every fiber $P_t$ at a single point;
 \item  $\Aut(X)_p$ contains no  automorphisms of type {\bf A} or {\bf B};
 \item  $\Aut(X)_p$  contains an  automorphism of type  {\bf C}; 
 \item there exists a bisection of $p$  
 %May 16: deleted ,i.e., a divisor $D\subset X$ 
 that intersects every fiber $P_t$ at two 
 %May 16: inserted
 distinct
 %%%%%%%%%%%%%
  points.\end{itemize}

 We use the fact that 
 %May 16: replace the by distinct
 distinct 
 %%%%%%%%%%%%%%%
 sections of a $\BP^1-$bundle over a torus  $T$  with $a(T)=0$ do not intersect,
 thus our example is impossible with $\dim(T)=1.$

 Let $S$ be a torus with    $$\dim(S)=n\ge 2,  \ a(S)=0.$$
   Let $\LL$ be a {\sl nontrivial} holomorphic line bundle over $S$ such that
 \begin{enumerate}
  \item $\LL\in \Pic_0(S)$;
  \item $\LL^{\otimes 2}=\mathbf{ 1}_S.$ 
  \end{enumerate}
Let $Y $ be the  total body of $\LL$ and $q:Y\to S$ the corresponding surjective holomorphic map.  
Consider   the   rank two vector bundle $E:=\LL\oplus \mathbf{ 1}_S$ on $S$    and let  $\ov   Y =\BP(E)$ be the projectivization of $E,$ let $\ov q:\ov Y \to S$ the holomorphic extension of $q$ to  $\ov   Y .$   
The holomorphic map $\ov q$  has precisely two sections, namely,  $D_0$ that is the zero  section of $\LL$ and $D_\infty=\ov   Y
\setminus Y.$  Since $\LL$  is  a nontrivial line bundle  and $a(S)=0,$ there are no other sections of $\ov q.$  

We may describe $Y$ in the following way (see \cite[Ch. 1, Sect. 2]{BL}). Let $S=V/\Gamma, $  where $V=\BC^n$ is  n-dimensional complex  vector space, $n=\dim(S)$ and $\Gamma$ is a discrete lattice of rank $2n$. 
Then   there exists a {\sl nontrivial} group homomorphism 
$$\xi  :\Gamma\to \{\pm 1\}\subset \BC^{*}$$
 such that $Y$ is the quotient  $(V\times \BC)/\Gamma$ 
%\begin{itemize}\item $\xi(\gamma^2)=1$ for all $\gamma \in\Gamma.$\item 
with respect to the action of group $\Gamma$  on $V\times \BC$  by automorphisms 
\begin{equation}\label{g}  
g_{\gamma}: (v,z)=(v+\gamma, \xi(\gamma)z) \ \forall \gamma \in\Gamma,  (v,z)\in V\times \BC_z .\end{equation} 
%for every $\gamma \in\Gamma$ and $(v,z)\in V\times \BC_z.$\item
%then $Y=(V\times \BC)/G.$ 
%\end{itemize}
We may extend the action of $\Gamma$ to $V\times\BP^1=V\times \bar{\BC}_z$ by the same formula \eqref{g}  and get $\bar{Y}=(V\times \bar{\BC}_z)/\Gamma$.

 Let us consider the following three holomorphic automorphisms of $\ov Y.$
 
 1)  The line bundles $\LL$ and $\LL^{-1}$ are isomorphic. Hence, there is a holomorphic {\sl involution map} $I_L:\ov Y\to \ov Y$ such that  $I_L(D_0)=D_\infty$ and $I_L\circ I_L=id.$   Automorphism $I_L$  may be included into the commutative diagram 
\begin{equation}\label{diagram100}
\begin{CD}
\ov Y@>{ I_L}>>\ov Y\\
@V \ov q VV @V\ov qVV \\
S  @>{ id}>> S
\end{CD}.
\end{equation}
In order to define $I_L$ explicitly, let us consider a holomorphic involution
\begin{equation}
\label{tildeL}
\tilde{I}_L: V\times \bar{\BC}_z \to V\times \bar{\BC}_z, \ (v,z) \mapsto  ( v, \frac{1}{z }).
\end{equation}
We have for all $\gamma\in\Gamma$
$$g_{\gamma}\circ  \tilde{I}_L (v,z)=(v+\gamma, \xi(\gamma)\cdot \frac{1}{z}),$$
 $$ \tilde{I}_L\circ   g_{\gamma}(v,z)=(v+\gamma, \frac{1}{\xi(\gamma)z})=g_{\gamma}\circ  \tilde{I}_L (v,z),$$
 since $\xi(\gamma)^2=1.$ In other words, $ \tilde{I}_L$ commutes with the action of $\Gamma$ and therefore descends  to the holomorphic involution of $(V\times \bar{\BC}_z)/\Gamma=\bar{Y}$ and this involution is our $I_L$.

 2) Let us choose $\gamma_0\in \Gamma$ such that
 \begin{equation}
 \label{gamma0}
 \gamma_0\not\in 2\Gamma, \ \xi(\gamma_0)=1.
 \end{equation}
 (Such a $\gamma_0$ does exist, since the rank of $\Gamma$ is greater than $1$.) Let us put
 $$v_0:=\frac{\gamma_0}{2}\in \frac{1}{2}\Gamma \subset V$$
 and consider an order 2 point
  $$P:=v_0+\Gamma\in V/\Gamma = S.$$ 
  %be a point of order 2 . 
  Then the translation map
  $$\TT_P: S \to S, \ s\mapsto s+P$$
   is a holomorphic involution on $S:  \TT_P^2=id$. Since $\LL \in \Pic^0(S)$, 
  translation $\TT_P$ induces a holomorphic involution   $I_P:\ov Y\to\ov Y$ \cite{Zar19}  that ``lifts'' $\TT_P$ and leaves $D_0$ and $D_\infty$ invariant.   The automorphism $I_P$  may be included in the commutative diagram 
\begin{equation}\label{diagram400}
\begin{CD}
\ov Y@>{ I_P}>>\ov Y\\
@V \ov q VV @V\ov qVV \\
S  @>{ \TT_P}>> S
\end{CD}.
\end{equation}
In order to describe $I_P$ explicitly, let us consider a holomorphic automorphism
% First, let us  choose  a preimage
% $v_0$  of $P$ in $V$.
% Then $\gamma_0=2v_0\in \Gamma$.  Choose $zeta=\sqrt{\xi(\gamma)}\in \BC^{*}$ and consider a holomorphic automorphism
 %In $Y=(V\times \BC)/\Gamma$ map  $I_P$ is induced by
\begin{equation}\label{ip} 
  \tilde{I}_P:V\times \bar{\BC}_z \to V\times \bar{\BC}_z,  \  (v,z)=(v+v_0, z). \end{equation} 
  Clearly, 
  $$\tilde{I}_P^2=g_{\gamma_0}$$
  (recall that $\xi(\gamma_0)=1$).
 For all $\gamma\in \Gamma$
 $$g_{\gamma}\circ  \tilde{I}_P (v,z)=(v+v_0+\gamma, \xi(\gamma)z),$$
 $$ \tilde{I}_P\circ g_{\gamma}(v,z)=(v+\gamma+v_0, \xi(\gamma)z)=(v+v_0+\gamma, \xi(\gamma)z),$$
 i.e., $ \tilde{I}_P$ and $g_{\gamma}$ do commute.  This implies that  $\tilde{I}_P$ descends  to the holomorphic involution of $(V\times \bar{\BC}_z)/\Gamma=\bar{Y}$, and this involution is our $I_P$.

3)    Let $\ov h\in \Aut(\ov Y)$ be the holomorphic involution that  acts as multiplication
 by $-1$ in every fiber of $\LL.$ (In notation of \cite{Zar19} $  \ov h=\mult(-1).$)
In $\bar{Y}=(V\times \bar{\BC}_z)/\Gamma$ map $\ov h $ is induced by the holomorphic involution
\begin{equation}\label{h} 
 \tilde h: V\times \bar{\BC}_z \to V\times \bar{\BC}_z, \  (v,z)=(v, -z),\end{equation} 
 which commutes with all $g_{\gamma}$. Indeed,
   for all $\gamma\in \Gamma$
 $$g_{\gamma}\circ \tilde h(v,z)=(v+\gamma, \xi(\gamma)(-z))=(v+\gamma, -\xi(\gamma)z),$$
 $$\tilde h\circ g_{\gamma}(v,z)=(v+\gamma, -\xi(\gamma)z)=g_{\gamma}\circ\tilde h(v,z).$$
% i.e., $h$ and $g_{\gamma}$ do commute.

 Let us show that $I_L,$ $I_Y$ and $\ov h$   commute.  It suffices to check that  $\tilde{I}_L$, $\tilde{I}_Y$ and $\tilde h$   commute, which is an immediate corollary of the following direct computations.
 $$\tilde{I}_L\circ\tilde h(v,z)=(v, \frac{1}{-z})=(v, -\frac{1}{z})=\tilde h\circ  \tilde{I}_L(v,z),$$
 $$\tilde{I}_L\circ \tilde{I}_P(v,z)=(v+v_0, \frac{1}{z})=\tilde{I}_P\circ  \tilde{I}_L(v,z),$$
 $$\tilde h\circ \tilde{I}_P(v,z)=(v+v_0, - z)=\tilde{I}_P (v,-z)=\tilde{I}_P\circ\tilde h(v,z).$$
% Thus, those maps commute on the open dense subset $\ov Y\setminus (D_0\cup  D_{infty}),$ hence, they commute on all $\ov Y.$

  Let us put now  $$\invl:=I_P\circ I_L: \ov Y\to \ov Y.$$  Then: 
  \begin{itemize}\item[1)] $\invl^2=id;$
  \item[(2)] $ \invl\circ\ov h=\ov h\circ \invl;$
   \item[(3)] $ \ov q\circ \invl=\TT_P\circ\ov q;$
       \item[(4)]  $\TT_P$ has no fixed points, thus $\invl$ has no fixed points; 
  \item[(5)] $\invl (D_0)=D_\infty;$
  \item[(6)] If $d_1,d_2\in D_0$ then $\invl(d_1)\ne d_2.$\end{itemize}
 
Let $X$ be the quotient of $\ov Y$ by the action of the order 2 group $\{id, \invl\}, $ and 
$\pi_Y: \ov Y \to X$ be the corresponding quotient map. Let 
$T$ be  the quotient of $S$  by the action of  the order 2 group $\{id, \TT_P\},$ and $\pi_S:S \to T$ be the corresponding quotient map. 
  Then $X$ and $T$ enjoy the following properties.
\begin{itemize}\item For any $x\in X$ there are precisely two points $y, \invl(y)$ in  $\pi_Y^{-1}(x)$.
\item For any $t\in T$ there are precisely two points $s, \TT_P(s)$ in  $\pi_S^{-1}(t)$.
\item  Both $\pi_Y:\ov Y\to X$   and   $\pi_S:S\to T$   are double unramified  coverings.
\item $X$ is a smooth complex manifold (by   (4)).
\item $T$ is a complex torus with $a(T)=0, \dim(T)=\dim (S)\ge 2$.
\item  It follows from (3) that there is  a holomorphic map  $p:X\to T$  such that the following diagram commutes.
  \begin{equation}\label{diagram500b}
\begin{CD}
\ov Y@>{\pi_Y }>> X\\
@V \ov q VV @VpVV \\
S  @>{\pi_S }>> T
\end{CD}.
\end{equation}
\item   If $\pi_S(s)=t\in T$ then     $p^{-1}(t)\sim\ov  q^{-1}(s)\sim \BP^1$.
\item  It follows from  (2) that there is a holomorphic map (pushdown) $h:X\to  X$  such that the following diagram commutes:
  \begin{equation}\label{diagram500}
\begin{CD}
\ov Y@>{\ov h}>>\ov Y \\
@V\pi_YVV @V\pi_YVV \\
X  @>{h}>> X
\end{CD}.
\end{equation}
\item  Thanks to (5), we have  $\pi_Y(D_0)=\pi_Y(D_\infty):=D;$
\item   Thanks to  (6), the restriction   $p \bigm |_D:D\to T$ is a double covering. \end{itemize}

 It follows that  $X$ is a $\BP^1$-bundle over $T,$  $D$ is a bisection of $p$  and $h$ is a nontrivial automorphism in $\Aut(X)_p$ of order $2$, whose set of fixed points coincides with $D.$
 
\begin{Lemma}\label{nosec} There is no section  of $p.$ \end{Lemma}

\begin{proof}
 Assume that $p$ has a section $\sigma:T\to X$.
 Let $\Sigma:=\sigma  (T)\subset X$ and $\Delta:=\pi_Y^{-1}(\Sigma).$
As we have already seen, both maps $\pi_S$ and $\pi_Y$ are double unramified covers. 
For every point $t\in T$ there are precisely two distinct points $s$ and $\invl(s)$ in  $\pi_S^{-1}(t)$, and 
there are precisely two distinct points  in $\pi_Y^{-1}(\sigma(t)), $  say, $y_t$ and 
$\invl(y_t).$  One of them is mapped  by  $\ov  q$ to $s,$  another to $\invl(s).$  It follows that for every $s\in S$ there is precisely one point in $\Delta\cap \ov  q^{-1}(s).$  Hence, $\Delta $ is a section of $\ov q.$
 By construction, 
      $\ov q$ has no sections except $D_0$ and $D_{\infty}.$  But  
       $\Delta$ cannot coincide with $D_0$ or  $D_{\infty}$ since $\pi_Y(\Delta)=\Sigma$ is a section of $p$ and $D$ is not.  The contradiction shows that section $\sigma:T\to X$  does not exist. \end{proof}

  Hence, $p$ has no sections and, therefore, there are no automorphisms of type {\bf A} and {\bf B} in $\Aut(X)_p.$

\end{document}